\documentclass[12pt,notitlepage]{article}
\usepackage{standalone}
\usepackage{amsmath}
\usepackage{amssymb}
\usepackage{amsbsy}
\usepackage{amsfonts}
\usepackage{titlesec}
\usepackage[all,cmtip]{xy}
\usepackage{amsthm}
\usepackage{mathrsfs}
\usepackage{geometry,graphicx,color}
\usepackage{mathtext}
\usepackage{hyperref}
\usepackage{booktabs}

\theoremstyle{plain}
\newtheorem{theorem}{Theorem}[section]
\newtheorem{lemma}[theorem]{Lemma}

\newcommand{\Pf}{\mathrm{Pf}}
\newcommand{\s}{\mathrm{\protect\overrightarrow{\mathrm{S}}}}
\newcommand{\xar}[1]{\ensuremath{\xrightarrow{#1}}}
\newcommand{\mf}[1]{\mathfrak{#1}}
\newcommand{\mc}[1]{\mathcal{#1}}
\newcommand{\mr}[1]{\mathrm{#1}}
\newcommand{\mb}[1]{\mathbf{#1}}
\newcommand{\R}{\mathbb{R}}
\newcommand{\Q}{\mathbb{Q}}
\newcommand{\Z}{\mathbb{Z}}
\newcommand{\C}{\mathbb{C}}
\newcommand{\T}{\mathbb{T}}
\newcommand{\U}{\mathrm{U(1)}}
\newcommand{\bs}[1]{\mathbf{#1}}
\newcommand{\on}[1]{\operatorname{#1}}
\newcommand{\la}{\langle}
\newcommand{\ra}{\rangle}	
\newcommand{\ol}[1]{\overline{#1}}
\newcommand{\ul}[1]{\underline{#1}}
\newcommand{\op}{\mathrm{op}}
\newcommand{\N}{{\pmb  \Delta}} %categorial simplex
\newcommand{\NN}{{\underline{\pmb \Delta}}} %categorial semi-simplex

\titleformat{\subparagraph}[runin]{\normalsize\bfseries}{\textbf{\arabic{subparagraph}.}}{1em}{}[]

\title{On local combinatorial formulas for Chern classes of triangulated circle bundle}
\author{ Nikolai Mnev\thanks{PDMI RAS;  Chebyshev Laboratory, SPbU. Supported by the Russian Science Foundation grant №14-21-00035 } \\ mnev@pdmi.ras.ru \\
	Georgy Sharygin\thanks{ITEP; MSU.
	Supported by the Russian Science Foundation grant
	№16-11-10069 and RFBR grant №14-01-00329} \\ sharygin@itep.ru}

\begin{document}
\maketitle
\begin{abstract}
	Principal circle bundle over a PL polyhedron can be triangulated and thus obtains combinatorics. The triangulation is assembled from triangulated circle bundles over simplices. To every triangulated circle bundle over a simplex we associate a necklace (in combinatorial sense). We express rational local formulas for all powers of first Chern class in the terms of mathematical expectations of parities of the associated necklaces. This rational parity is a combinatorial isomorphism invariant of triangulated circle bundle over simplex, measuring mixing  by triangulation of the circular graphs over vertices of the  simplex. The goal of this note is to sketch the logic of deduction these formulas from Kontsevitch's cyclic invariant connection form on metric polygons.
	
\end{abstract}
\tableofcontents
\section{Introduction}
\setcounter{subparagraph}{-1}
\subparagraph{}
After submitting this note we discovered that the main computation in Sections \ref{geometry},\ref{algebra} is equivalent to computation in \cite[Sections 1,2]{Igusa2004} of universal combinatorial cochains on cyclic category out of universal cyclic connection form. Still we have an accent on geometry and combinatorics of triangulations.
\subparagraph{}
Circle bundle $\T \xar{} E \xar{p} B$ (\cite{Chern1977}) is a principal fiber bundle with a commutative Lie structure group $\T = \R/\Z \approx \U$. There is a classic chain of  homotopy equivalences
\begin{equation}\label{equiv}
B\T \approx B\U \approx B\mr O(2) \approx \C P^\infty \approx K(\Z,2).
\end{equation}
Thus isomorphism classes of circle bundles over $B$ are in one-to one correspondence with classes of complex line bundles, classes of  orientend 2-dimensional real linear vector bundles and elements of 2-dimensional integer cohomology group $H^2(B;\Z)$. The class of bundle $c_1(p) \in H^2(B;\Z)$ is called its first Chern or Chern-Euler class. All characteristic classes of $p$ are powers $c^h_1(p) \in H^{2h}(B;\Z)$.

If base $B$ can be triangulated then $p$ can be triangulated over a subdivision of given triangulation. It is not true that the base $B$ can always be triangulated (see \cite{Manolescu2014}) but it happens in most interesting cases. See \cite{Goresky1978,Verona1980} for general triangulation and bundle triangulation theory. We are interested in combinatorics of triangulations in connection with integer and rational local combinatorial formulas for Chern class and its powers. Triangulated bundle obtains structure of oriented PL bundle with fiber $S^1$. Old folklore
\begin{equation}\label{pl/u}
\on{PL}(\s)/\U \approx *
\end{equation} theorem completes sequence (\ref{equiv}) by ``$B\on{PL}(\s)\approx ... $''. Therefore combinatorics of triangulation contains full knowledge about characteristic classes.
\subparagraph{}
Our initial intention is to find some more hints for correct combinatorics in the classical problem of detecting  local combinatorial formulas for characteristic classes (\cite{GGL},\cite{GM1992},\cite{Gaifullin:2005}). It is well known that the problem faces many troubles in a  fruitful way. Particularly, Chern-Simons theory was a by-product as it is mentioned in the first lines of \cite{CS74}. Configuration spaces, matroidal point of view on the problem (\cite{GGL},\cite{GM1992})  faces algebro-geometric universality of moduli spaces of configurations in a way especially interesting for the first author. Rational Euler \cite{Rocek1991} and Stiefel-Whitney \cite{Halperin1972} classes of tangent bundle have a clear local combinatoric nature (but in the last case up to now there is only a not very enlightening proof). From the combinatorial point of view local formulas are very interesting combinatorial functions, universal cocycles, associated to elementary families of cell complex reconstructions. Such as chains of abstract subdivisions or chains of simple maps, abstract mixed subdivisions (multi-simplicial complexes) chains in MacPhersonian etc.

In the line of simple examples obviously there should be Chern classes of triangulated (or in some other way locally combinatorially encoded) circle bundles. The setup for combinatorics of circle bundle and an outline  of construction for the rational simplicial local formulas was presented by I. Gelfand and R. MacPherson in 14 lines around Proposition 2 \cite[p.306]{GM1992}. There are deep formulas in differential situation  when the bundle is encoded by pattern of fiberwise  singularities of Morse function on total space \cite{IgusaKlein1993,Kazarian:1998,Kazarian97}. These last formulas connect naturally the problem to higher Franz-Reidemeister torsion \cite{Igusa2002} and cyclic homology.
The role of line bundles in geometry of characteristic classes and related mathematical physics  is special (\cite{Brylinski08}), so perhaps a reasonable idea is to understand the case of circle bundles slightly better. One day to our surprise we discovered   that canonically looking local formulas for all characteristic cocycles of triangulated circle bundle trivially pop out from associating to triangulated circle bundle Kontsevich's connection \cite[p.8]{Kontsevich92} on metric polygons. The answer is expressed through mathematical expectations of parities of necklaces associated with triangulation. Below we will present this  ``parity local formulas" and a sketch of construction.

\subparagraph{Plan.}
In Section \ref{loc_form} we introduce abstract simplicial circle bundle (s.c. bundle) and describe Gelfand-MacPherson's  setup for simplicial local formulas in the case of circle bundles. In Section \ref{scbcw} we associate to a s.c. bundle cyclic diagram of words on the base. In Section \ref{rplf} we present rational parity formula for all powers of first Chern class and formulate main Theorem \ref{thm}. In Section \ref{geometry} we introduce on geometric realization of s.c. bundle canonical metric of ``geometric proportions" $gp$. With this special metric $gp$ the bundle becomes piecewise-differential principal circle bundle. It has gauge transition transformations described by functions encoded by matrices of words as linear operators (analogs of permutation matrix). The circle bundle canonically appears as a pullback of formal universal circle bundle over Connes' cyclic simplex by very special classifying map.  The classifying map is described by matrices of words as linear operators. In Section \ref{algebra} we describe Kontsevich's connection form on mteric polygons as universal cyclic invariant connection  on universal circle bundle over Connes' cyclic cosimplex. The pullback of the connection is PD connection on geometric realisation of s.c. bundle with metric $gp$. We compute pullbacks of universal cyclic characteristic forms using matrix maps and, using Okuda's sum of minors - Pfaffian identity, obtain rational parity coefficients.
In Section \ref{proof} we assemble the proof of Theorem \ref{thm}.
In Section \ref{notes} we comment omissions and possible outputs.

\subparagraph{}
The first author is deeply grateful to Peter Zograf for pointing on Kontsevich's connection form.  Authors thanks for hospitality
Oberwolfach Mathematical Institute and IHES where they had an opportunity to work together.

\section{Local simplicial formulas for circle bundles} \label{loc_form}

\subparagraph{Triangulation of circle bundle.}
If a map of finite abstract simplicial complexes $\mf E \xar{\mf p} \mf B$ triangulates circle bundle $p$ then we can suppose that geometric realization  $|\mf B| = B$ and the triangulation is a set of data $(\mf p, p,h)$:
\begin{equation}\label{triang}
\xymatrix{
{\mf E} \ar[d]^{\mf p} & E \ar[d]^p & {|\mf E|}\ar[rr]^h \ar[dr]_{|\mf p|} & & E \ar[dl]^p \\
\mf B                   & {|\mf B|}       &  & {|\mf B|} &
}
\end{equation}
where $|\mf E|\xar{h} E$ is a fiberwise homeomorphism commuting with $|\mf p|$ and $p$. The map $|\mf p|$ is a PL fiber bundle. On the other hand, from the point of view of $\mf p$ triangulation homeomorphism $h$ is equivalent to introducing on PL bundle $|\mf p|$ with fiber $S^1$  orientation and  continuous metric such that for any $x\in |\mf B|$ the fiber over $x$, oriented PL circle $|\mf p|^{-1}(x)$, has perimeter equal to one. From this point of view, triangulation homeomorphism $h$ provides $p$ by PL structure related to very special system of local sections. We suppose that simplicial complex $\mf B$ is \textit{locally ordered}, i.e. its simplices have total orders on vertices and face maps are monotone injections. That is to say $\mf B$ is a finite semi-simplcial set with an extra property that each simplex is determined by its vertices. Local order makes simplicial chain and cochain complexes available.

Simplicial bundle $\mf p$ in (\ref{triang}) can be assembled as a colimit from subbundles  over base ordered simplices using simplicial face transition maps. Disassembly on bundles over simplices commute with geometric realization. Therefore subbundle  of $\mf p$ over simplex  triangulates circle bundle over geometric simplex, which is trivial bundle. Orientation of trivial circle bundle selects preferred generator in 1-homology of total space. Hence subbundles of $\mf p$ over simplices are equipped with orientation class in simplicial homology in such a way that boundary transition maps of subbundles sends generator to generator. This assembles to constant fiber orientation local system on $\mf B$.
\subparagraph{Basic simplicial notations.} \label{simp}
We denote by $\N$ category of finite ordinals $[0],[1]...$
Non-strictly  monotone maps between them are called operators; injections are called face operators surjections are called degeneracy operators. We denote by $\NN$ subcategory of $\N$  having face operators only. Category of finite ordinals
can be presented by generators and relations. The standard generators of $\N$
are elementary boundary operators $[k-1]\xar{\delta_j} [k], j=0,...,k$, the monotone injections missing element $j$ in the image and $[k+1]\xar{\sigma_j}[k],j=0,...,k$, the monotone surjections sending elements $j,j+1$ in domain to the element $j$ in codomain. Category $\NN$ is generated only by boundaries.  We denote by $\la k \ra$ ordered combinatorial simplex: simplicial complex formed by all subsets of $[k]$. For an operator $[m] \xar{\mu} [k]$ We denote by $\la m \ra \xar{\la \mu \ra} \la k \ra$ the induced simplicial map of combinatorial simplices. Standard geometric $k$-simplex with ordered vertices is denoted  by $\Delta^k$. Geometric realisation of $\la \mu \ra$ is
the map $\Delta^m \xar{|\mu|} \Delta^k$ written in barycentric coordinates
$$\text{for $i=0,...,k$ }|\mu_i| (t_0,....t_m) =
\begin{cases}
0 & \text{if $i \not\in \on{im} \mu$}\\
\sum_{j\in \mu^{-1}(i)} t_j &  \text{if $i \in \on{im} \mu$}
\end{cases}  $$
\subparagraph{Simplicial circle bundles.} \label{ccb}
We shall call \textit{elementary simplicial circle bundle} (elementary s.c. bundle) a map $\mf R \xar{e} \la k\ra $ of simplicial complex $\mf R$ onto ordered simplex $\la k\ra$, whose geometric realization $|\mf e|$ is a trivial $PL$ fiber bundle over geometric simplex with fiber $S^1$, equipped with fixed orientation 1-dimensional homology class of $\mf R$. Here is a picture of elementary s.c. bundle (Fig. \ref{eb}).
\begin{figure}[h!]
	\begin{center}
		\includegraphics[width=3.0in]{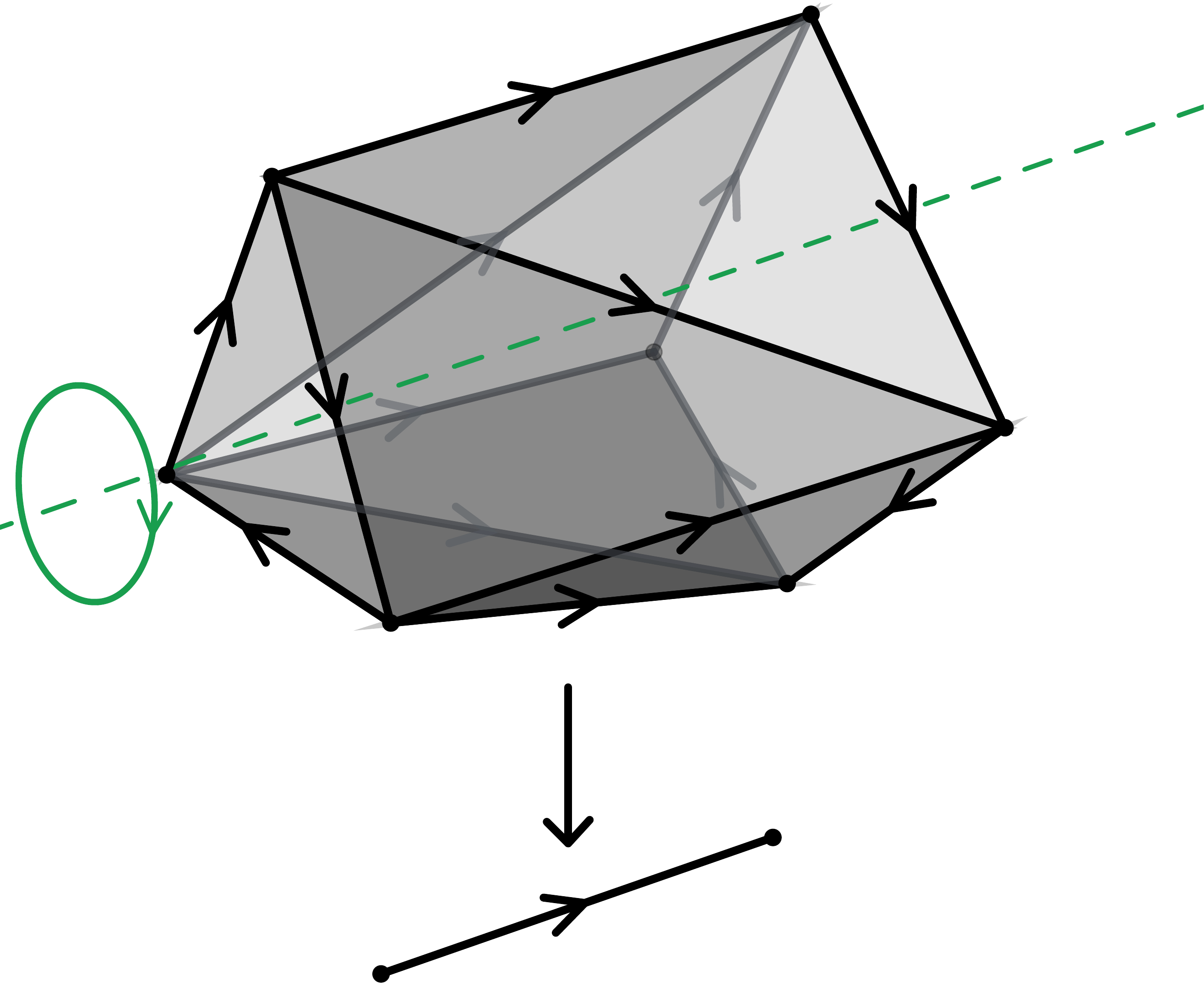} %\includegraphics[width=2.0in]{ebundl1.pdf}
		\caption{ Elementary simplicial circle bundle \label{eb}}
	\end{center}
\end{figure}
We define \textit{boundary map} $\delta^* \mf e \xar{\delta_* e}\mf e$ of elementary s.c. bundles over simplicial boundary  $[m]\xar{\delta}[k]$ to be a pullback diagram
$$\xymatrix{
\delta^*\mf R \ar[r]_{\delta_* \mf e}  \ar[d]^{\delta^* \mf e}& \mf R \ar[d]^{\mf e} \\
\la m \ra \ar[r]^{\la \delta \ra} & \la k \ra
}$$
We suppose that $\delta_* \mf e$ sends orientation class of $\delta^*\mf R$ to orientation class of $\mf R$.
Let $\mf B$ be a locally ordered finite simplicial complex. We call \textit{simplicial circle bundle (s.c. bundle)} $\mf E \xar{\mf p}\mf B$ a map of finite simplicial complexes $p$ that has a geometric PL fiber bundle with fiber $S^1$ as its geometric realization and equipped with a fixed local system that gives fiber orientation on elementary subbundles.

Simplicial circle bundles always triangulate some principal circle bundles. To build a circle bundle, triangulated by $\mf p$, it is sufficient to choose a continuous metric on geometric realisation $|\mf E|$ such that all fibers of $|\mf p|$ were circles with unit perimeter. Transition maps become orientation preserving isometries and hence they will be $\T$ - transition maps.  One possible choice of such metric is to normalize fiberwise standard flat metric on simplexes in $|\mf E|$. Taking in consideration the simple fact (\ref{pl/u}), we see that all circle bundles triangulated by $\mf p$ are isomorphic.

\subparagraph{Gelfand-MacPherson's rational simplicial local formulas} \label{GM}
Denote by ${\mf R}^c(\s)$ the semi-simplicial set of ismorphism classes of all elementary simplicial circle bundles. Elements of ${\mf R}_k^c(\s) $ are combinatorial isomorphism classes of elementary s.c. bundles over $k$-dimensional oriented combinatorial simplex. Boundary map is generated by elementary s.c. bundles boundary  over base simplex  boundary. Boundaries are well-defined, sending isomorphism classes to isomorphism classes.  Superscript ``$c$" stays for the fact that we are considering triangulations by classical simplicial complexes.  With the bundle $\mf p$ a map of semi-simplicial sets $\mf B \xar{\mf G_{\mf p}} \mf{ R}^c(\s)$ is associated, sending base simplex $U \in \mf B$ to the combinatorial isomorphism class of elementary subbundle $\mf p_U$ over that simplex. The map $\mf G_{\mf p}$ forgets a part of information about the bundle. Since elementary s.c. bundle can have nontrivial automorphisms one can not recover all boundary transition functions from $\mf G_{\mf p}$ and hence generally one can not recover from $\mf G_{\mf p}$ the entire $\mf p$ up to isomorphism.

 \textit{Rational simplicial local formula for Chern class}
$C_1^h$ is a $2h$-cocycle on $\mf{ R}^c(\s)$ represented as a rational combinatorial function of elementary bundles over $2h$ simplex, such that pullback of this cocycle under the map
$\mf B \xar{\mf G_{\mf p}} \mf R^c(\s)$ is a rational simplicial $2h$-Chern cocycle of $\mf p$. To put it simple, the value of cocycle on the ordered simplex $U^{2h}$ of $\mf B$ should be an automorphism invariant $C^{2h}_1(\mf p_u)\in \Q$ of  subbundle $\mf p_U$ over that simplex.

Due to the mentioned forgetful nature of $\mf G_{\mf p}$ it is not obvious that such universal cocycles exist, however the rational universal cocycles indeed exist  by Proposition 2 of \cite{GM1992}: it is speculated there that the transgression of rational fiber coorintation class in Serre spectral sequence of s.c. bundle can be expressed as a local rational formula involving combinatorial Laplacian. The result was never calculated, but clearly it ends up in a slightly more complicate automorphism invariant function of associated 3-necklace for 1-st Chern class than our rational parity. The local formulas for powers one can always express automatically by using the \v Cech-Whitney formula for cup product on simplicial cochains. This will result in certain not very transparent, but purely combinatorial concrete rational functions $\mbox{}^{GM}\!C_1^{2h}(\mf e)$ of elementary s.c. bundles.

\section{Simplicial circle bundles and cyclic words.} \label{scbcw}

\subsection{Half of Connes' cyclic category}
Connes' cyclic category  $\N C$ (see \cite[Ch 6.1]{Lo}) has finite ordinals as objects and morphisms are generated by all operators from $\N$ and all cyclic permutations of finite ordinals. By this definition $\N C$ contains $\N$ as a subcategory.
For cyclic category we use new  notations for standard simplicial generators since  we will need a two-parametric cyclic-simplicial structures on words.  The simplicial generators are
\begin{equation}\label{}
\xymatrixcolsep{3pc}\xymatrix{
{[n-1]}\ar@/^/[r]^-{\partial_j^{n-1}} & \ar@/^/[l]^-{s_i^n} [n]
}, j=0,...,n; i=0,...,n-1
\end{equation}
Apart of simplicial generators,
$\N C$  has generator $[n]\xar{\tau_n}[n]$ acting on $[n]$ by rule  $ \tau_n(i) = (i-1) \mod (n+1)$, i.e. as the permutation which is  left cyclic shift by one; $\tau_n^j(i) = (i-j) \mod (n+1)$. In $\N C$ there is an ``extra degeneracy", the surjection $[n] \xar{s^n_n}  [n-1]$
which does not exist in $\N$ but exist in $\N C$. It is defined as $s^n_n = s^n_0\tau^{-1}_n$.

\subparagraph{Duality.} \ There are two subcategories in $\N C$. One is $\NN C$, the subcategory
generated by all boundaries $\partial_i^n$  and cyclic permutations $\tau_n$. Another category $\overline{\pmb \Delta}C$ is generated by all standard monotone degeneracies $s^n_i$, \emph{extra degeneracy} $s_n^n$ and  cyclic shifts $\tau_n$.

The category $\N C$ is self-dual, i.e there is a categorial isomorphism
$\N C \xar{\bullet^\op}\N C^\op $. Since $\N C$ has automorphisms the duality involution is not unique.
With the extra degeneracy  a duality can be presented on generators as follows:
for $i \in [k]$ and $[k-1] \xar{\partial^{k-1}_i} [k]$;  the dual map is   $$\partial_i^\op = ([k] \xar{s^k_i} [k-1]);
\tau_k^\op = \tau_k^{-1}.$$ The duality interchanges $\NN C$ and $\overline \N C$.
\subparagraph{} \label{cyl} The duality has remarkable graphical interpretation in the terms of cylinder of ``simple map" of oriented graphs, oriented cycles (it can be  a loop) having a vertex fixed.

The map of oriented graphs is a map sending vertices to vertices, arcs to vertices or arcs in a way that incident vertex and arc go either to the same vertex or to incident vertex and arc. Oriented graphs can be identified with 1-dimensional semi-simplicial sets, maps of graphs with singular maps of semi-simplicial sets. Graph $g$ having geometric realization $|g|$, map of graphs $g_0 \xar{f} g$ also obtains a geometric realization $|g_0|\xar{|f|}|g|$ and we can consider cylinder $Cyl(|f|) \xar{}{[0,1]}$ of this map. The cylinder has natural structure of 2-dimensional cell complex composed from solid triangles and quadrangles and having cellular projection on interval $[0,1]$. The map $f$ is called \textit{simple }\cite{waldhausen2013} if for any point $u \in |g|$ the preimage $|f|^{-1} u \subseteq |g_0|$ is contractible.
Let both graphs $g_0,g_1$ be  oriented cycles. In this case the map is simple if preimage of every vertex is an embedded interval or, dually, any arc has a single preimage. The fundamental observation is that the map is simple iff projection $Cyl (f) \xar{} [0,1]$ is a trivial PL $\s$-fiber bundle. In big generality this is Cohen's theory of cylinder of PL map \cite{Cohen:1967}.
\begin{figure}[!h]
	\begin{center}
		\includegraphics[width=3.5in]{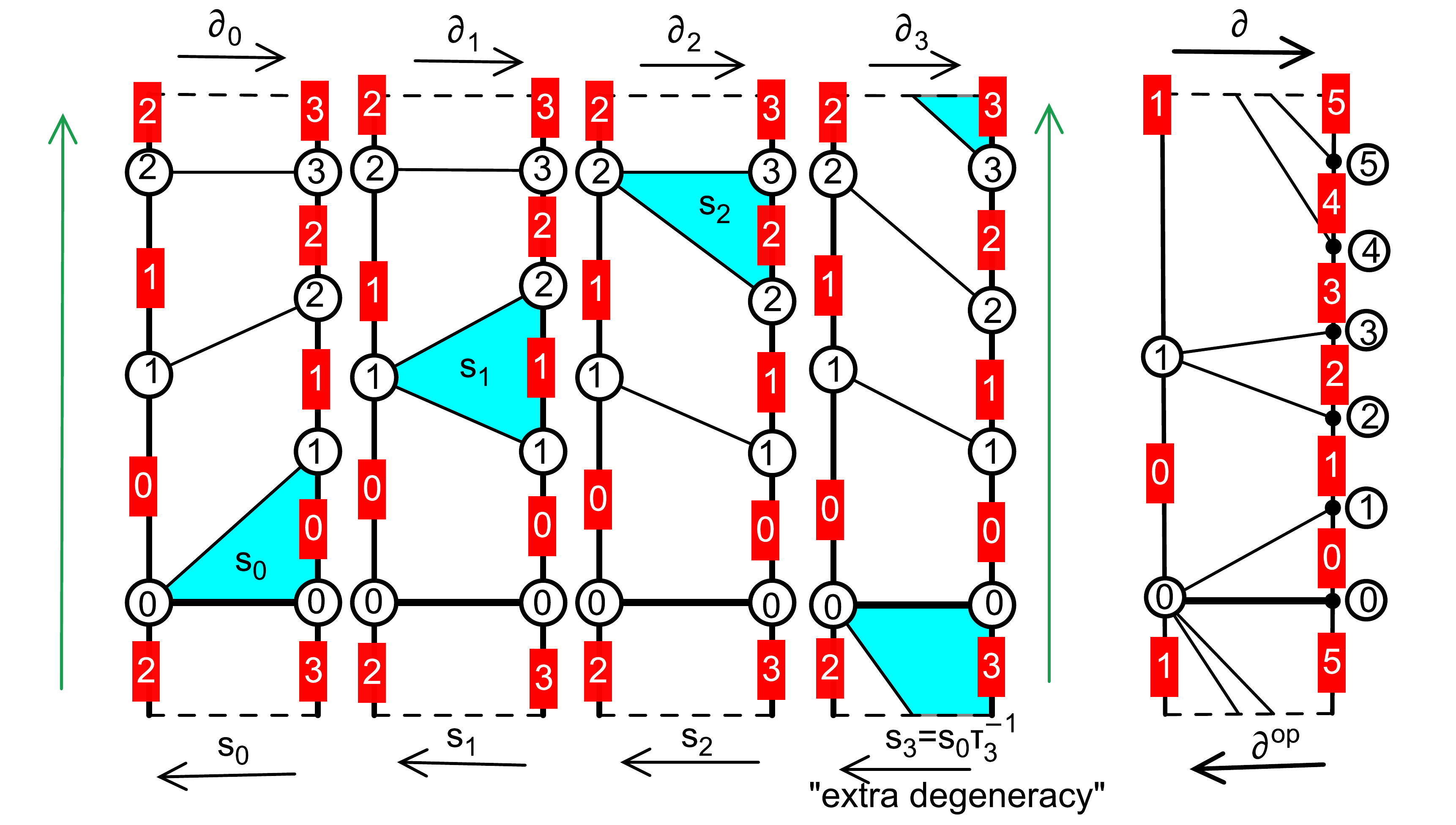}
		\caption{ \label{cycdual}}
	\end{center}
\end{figure}
 Let us have a fixed vertex $s_0$ on oriented cycle $g_0$. Then orientation creates a linear order on vertices and arcs of $g_0, g$. We suppose that in this order $s_0, f(s_0)$ are minimal. Also we suppose that the position of an arc in this ordering is equal to the position of its tail. In the terms of the orders, a map $f$ is simple iff the map $f$ is a $\ol \N C$-morphism $\ol f$ on ordered  vertices. If it the case  than dual of $\ul \N C$ morphism $\ol {f}^\op$ is a monotone injection on ordered arcs sending an arc to its unique preimage. On Fig \ref{cycdual} we depicted this duality on cone for generators and for general case.

For a  boundary
$([k]\xar{\partial}[m] )\in \mr{Mor}\NN C$ the general formula for dual cyclic degeneracy $([m]\xar{\partial^\op}[k]) \in Mor\overline\N C$ is as follows:
\begin{equation}\partial^\op(i)=\begin{cases}0 &\text{ if } \forall j: \partial (j) < i\\
\min_{\partial (j)\geq i } j & \text{ if } \exists j \text{ such that } \partial(j) \geq i \end{cases}\label{dual}\end{equation}
It can be computed graphically using the cylinder of simple map or inductively on generators.
\subparagraph{Cyclic-mono factorisation.}
Any morphism $[k] \xar{\varrho} [m]$ in $\N C$ has a unique factorization into cyclic permutation followed by monotone $\N$-operator.
The same is true for $\NN C$ and $ \overline{\N}C$.
This is Connes' theorem (\cite[Theorem 6.1.3]{Lo}).   We will fix a  simplified  formula for this decomposition for $\NN C$ case.
\begin{lemma}
	For any boundary map $[k] \xar{\partial} [m]$ and  cyclic permutation $[m]\xar{\tau^i_m} [m]$ there are unique boundary $[k] \xar{(\tau^i_m)^* \partial} [m] $ and cyclic permutation $[k]\xar{(\partial^* \tau_m^i)_k} [k]$ such that the diagram is commutative
	\begin{equation*}
	\xymatrixcolsep{4pc}\xymatrix{
	{[k]} \ar[d]_{\partial^*\tau^i_m} \ar[r]^{\partial} & {[m]} \ar[d]^{\tau_m^i} \\
	{[k]} \ar[r]_{(\tau_m^i)^* \partial} & {[m]}
	}
	\end{equation*}
	
	The following formula holds:
	\begin{equation}\partial^* \tau_m^i = \tau^{\partial^\op(i)}_k,\label{dualdec}\end{equation} where $\partial^\op$ defined in (\ref{dual}).
\end{lemma}
\begin{proof}
	This is a part of Connes' theorem, and it is proved by a part of Connes' proof
	which in our case is consideretion of the circular mapping cylinder of the map $\partial^\op$  with section fixed, which we discussed earlier. The map $\tau^i_m$ acts as  changing zero section of the cylinder.
\end{proof}

\subsection{Simplicial circle bundles and cyclic diagrams of  words on the base}
To a simplicial circle bundle defined in \S \ref{ccb} with fixed combinatorial sections $\bs S_0$ over simplices of base we associate a cyclic diagram of words on the base.

\subparagraph{Words and necklaces.} \label{cw}
\label{neck}
\textit{ Word of length $n+1$ in an ordered alphabet having $k+1$ elements}, $k \leq n$ is
any surjective map $[n]\xar{w} [k]$ (the map is not required to be monotone). Cyclic group $\Z/(n+1) \Z$ acts by cyclic permutations on $[n]$.  We can extend this cyclic action to words: just put $\tau_n w(i) = w(\tau_n(i)) $. Orbit of a word under cyclic permutations is called ``circular permutation" or \textit{``oriented necklace"} with $n+1$ beads colored by $[k]$.
We consider cyclic shift of word as a morphism between words
\begin{equation}\label{wrot}
\xymatrix{
{[n]} \ar[rr]^{\tau_n^i} \ar[dr]_{(\tau_n^i)^* w} & & {[n]} \ar[dl]^w \\
&{[k]} &}
\end{equation}
If we have a word $w$ and a boundary operator, i.e. a monotone injection $[k_0]\xar{\delta}[k]$, then a unique pullback is defined
\begin{equation}\label{wbound}
 \xymatrix{
{[n_0]} \ar[d]_{\delta^* w} \ar[r]^{\partial_{\delta}} &{[n]} \ar[d]^w  \\
 {[k_0]} \ar[r]^\delta &  {[k]}}
\end{equation}
where $\partial_{\delta}$ is the boundary operator on finite ordinals induced by boundary $\delta $ on codomain of words. To say it differently, the word $\delta^* w$ is a word obtained from $w$ by deleting all letters not in the image of $\delta$ and $\partial_{\delta}$ is the embedding of $\delta^* w$ into $w$ as a subword.

Define \textit{cyclic category of words} $C \mc W$ with words as objects, and morphisms \textit{cyclic morphisms of words}, i.e. $\NN C$-morphisms of words generated by cyclic shifts (\ref{wrot}) and boundaries (\ref{wbound}) (i.e. $w\mapsto \delta^*w$) over  boundaries on alphabet. Category $C\mc W$ has two projection functors
$$\NN C \xleftarrow{\on{Dom}} C\mc W \xar{\on{Codom}} \NN,$$
the ``domain" projection to $\NN C$ and ``codomain" projection to $\NN$. Codomain projection makes $C \mc W$ the category fibered in commutative groupoids over $\NN$. As in the category $\NN C$ any cyclic morphism of words has unique decomposition into cyclic shift followed by boundary.
\subparagraph{Words as matrices} \label{mat}
Words has associated matrices and we may consider these matrices as linear operators. This is our key trick.

Let $\mc L(w)$ be the $[n]\times [k]$ matrix
\begin{equation}\label{mw}
\mc L_i^j(w) =
\begin{cases}
1 & \text{if $w(i)=j$} \\
0 & \text{if $w(i)\neq j$.}
\end{cases}
\end{equation}
Put $m_j = \# w^{-1}(j)$, the number of times letter $j \in [k]$ appears in the word $w$. The following is the definition of the normalized by columns matrix of word $w$:
\begin{equation}\label{rmw}
\ol{\mc L}_i^j(w) =
\begin{cases}
\frac{1}{m_j} & \text{if $w(i)=j$} \\
0 & \text{if $w(i)\neq j$.}
\end{cases}
\end{equation}
Sums of elements in columns of $\ol{\mc L}$ are all equal to 1, different columns of $\mc L$ and $\ol{\mc L}$ are orthogonal.
Cyclic shift $(\tau^i_n)^*$ of words goes to cyclic permutation of rows of matrices, boundary operation $\delta^*$ on words corresponds to deletion of columns in matrices with numbers not in image of $\delta$ followed by deletion of zero rows in what remains. Opposite insertion of matrix as a submatrix with adding zero rows is described by $\partial_\delta$.

\subparagraph{Geometric fiber of elementary c.c. bundle; 0-,1-sections; order out of orientation.}
\begin{figure}[h!]
	\begin{center}
		\includegraphics[width=5.0in]{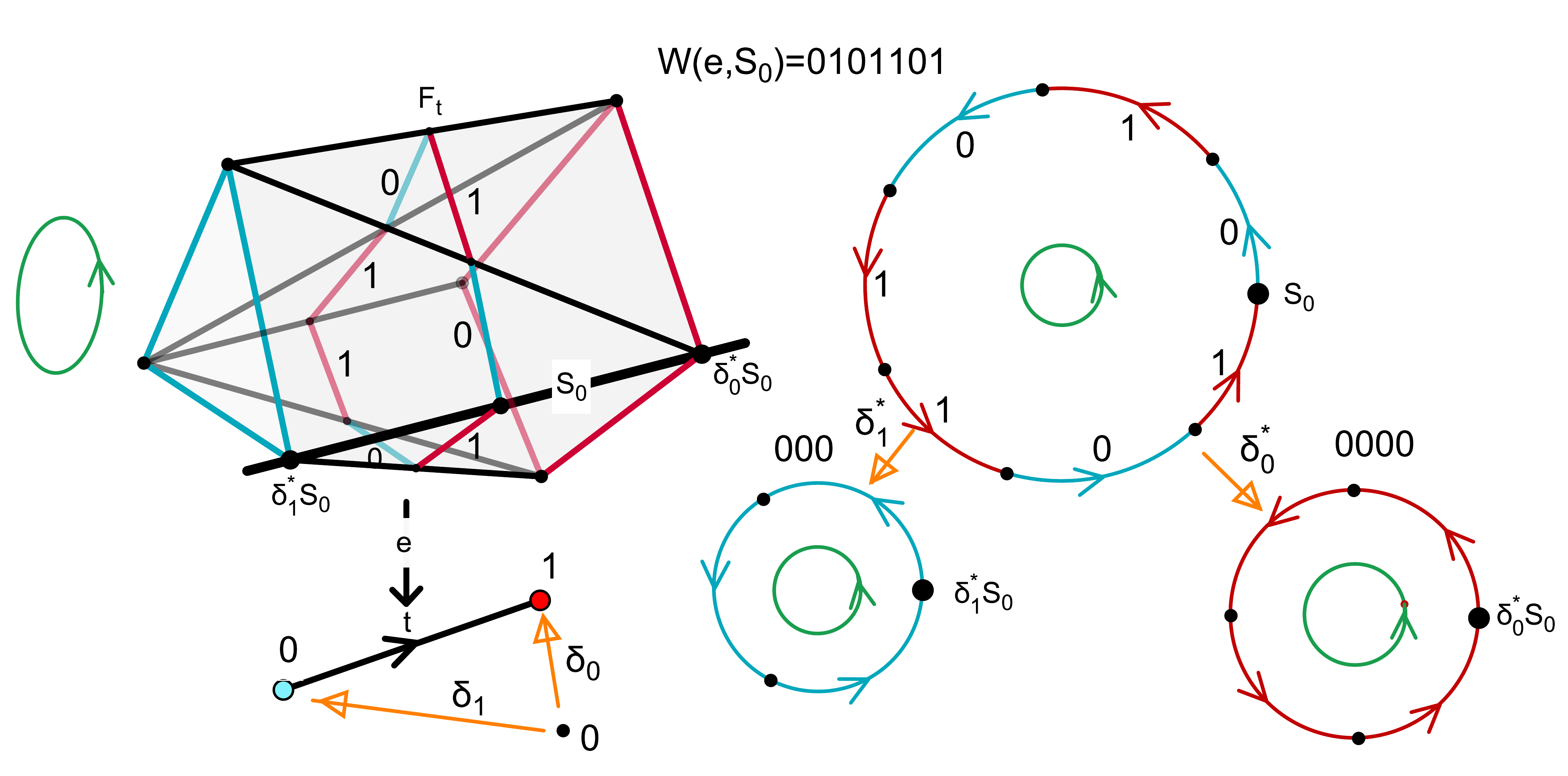}
		\includegraphics[width=2.5in]{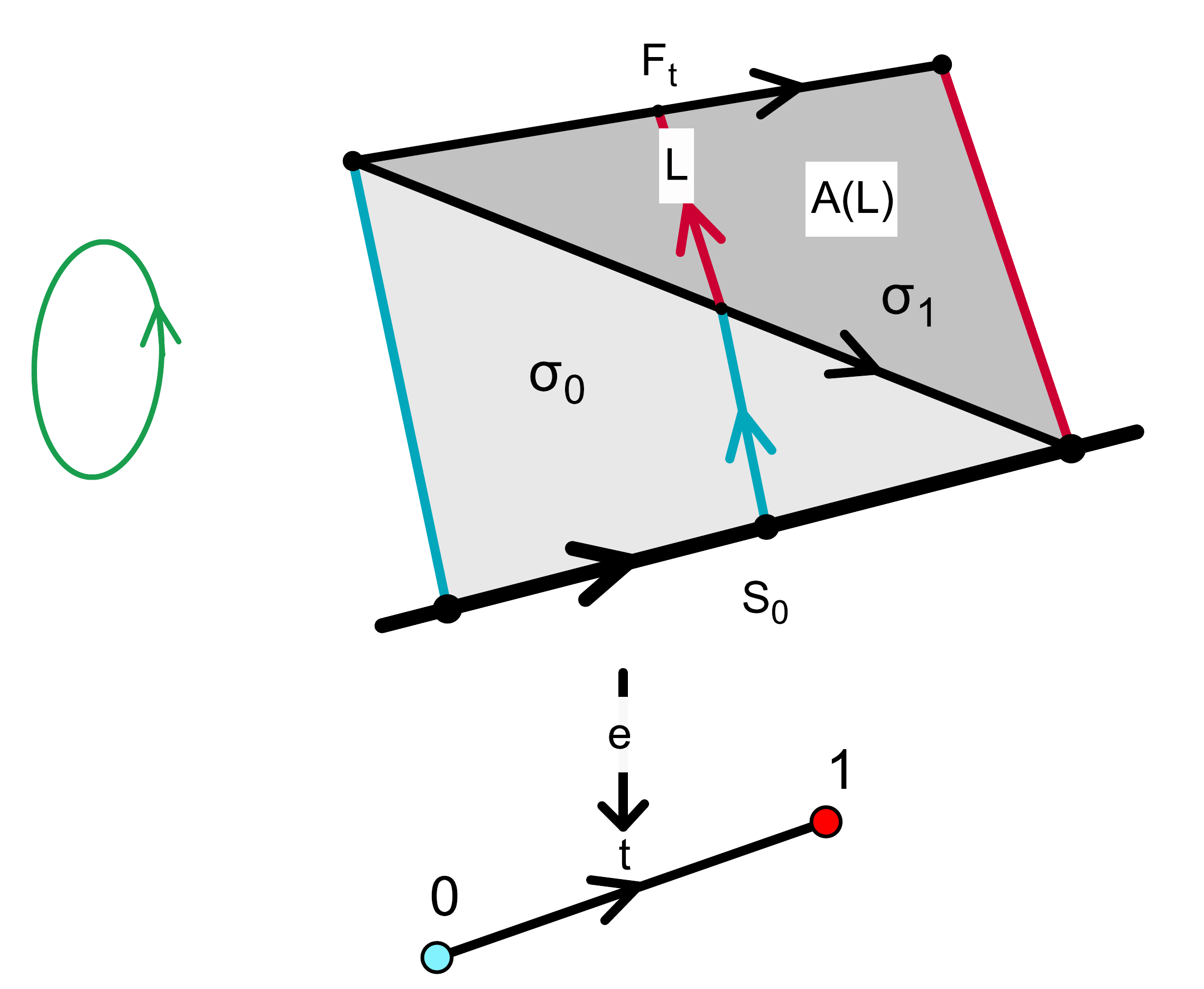}
		\caption{ \label{eb2}}
	\end{center}
\end{figure}

Let $\mf R \xar{\mf e} \la k \ra$ be an elementary s.c. bundle over ordered abstract $k$-simplex with $k+1$ vertices, $|\mf R|\xar{|\mf e|}\Delta^k$ its geometric realization, which is by definition an oriented trivial PL $S^1$-bundle over $\Delta^k$.  Geometrical simplices of $|\mf R|$ projecting epimorphically on $\Delta^k$ can have dimensions only $k$ and $k+1$. Call such a $k$-simplex a \textit{0-section} of $|\mf e|$ and a $k+1$-simplex a \textit{ 1-section } of $|\mf e|$. There are corresponding \textit{combinatorial 0-,1-sections}, the simplices in $\mf e$ that project epimorphically on $\la k \ra$.

Take geometric fiber $F_t$ of $|\mf e|$ over an interior point  $t\in \on{int}\Delta^k$. The fiber $F_t $ is a broken PL circle. Each arc $L$ of $F_t$ is an  intersection of the fibre with some $k+1$-simplex, $1$-section $A(L)$ of $|\mf e|$. The arc $L$  has as end vertices intersections with geometric $0$-sections, two faces of $A(L)$.  Broken circle $F_t$ has equal number of arcs and vertices, let this number be equal to $n+1$. This means that there are $n+1$ $0$- and $1$-sections of $|\mf e|$ and $\mf e$. Fix some $0$-section, denote it by $S_0$ and denote by $S_0(t)$ the corresponding vertex in $F_t$. Orientation of bundles $\mf e$, $|\mf e|$ creates orientation of circle $F_t$. Therefore all arcs $L$ obtain heads and tails, each vertex of $F_t$ obtains its input and output arcs. This generates a linear order on vertices of $F_t$ depending on the choice of the fixed vertex $S_0(t)$. In this order on any arc its tail precedes its head and $S_0(t)$ is the minimal vertex. The order on vertices creates a linear order on arcs in which an arc obtains the number of its tail. Therefore we obtain on $F_t$ $n+1$ ordered vertices $S_0(t),...,S_n(t)$ and $n+1$ ordered intervals $L_0(t),...,L_n(t); L_i(t)=[S_i(t),S_{i+1 \mod (n+1)}(t)],i=0,...,n$. Therefore the $1$-section $A(L_i)$ also obtains number $i$ in the order and we put $A_i = A(L_i)$.
\subparagraph{Word associated to elementary s.c. bundle with fixed section.} \label{word}
Every $1$-section $A_i$ of $|\mf e|$ is the domain simplex of a subbundle of type $\Delta^{k+1} \xar{|\sigma_j|}\Delta^k$  which is the geometric realization of an elementry simplicial degeneracy $\sigma_j$  of ordered simplices $\la k+1 \ra\xar{\la \sigma_j \ra} \la k\ra$, corresponding to the elementary degeneracy operator $[k+1]\xar{\sigma_j} [k]$ (\S\ref{simp}). Simplicial map $|\sigma_j|$ (Fig. \ref{simplex})
shrinks edge of $A_i=\Delta^{k+1}$ with vertices $v_{j},v_{j+1}$ to the single vertex $j$ of the base $\Delta^k$. Here  $j$ is a function of $A_i$ and therefore it is a function of number $i$. We denote this function as $j=\mc W(\mf e,S_0)(i)$. The function $\mc W(\mf e,S_0)$ is surjective and independent of $t \in \on{int} \Delta^k$, therefore it is a \textit{word (\S\ref{neck}), associated to the elementary s.c. bundle $\mf e$ with fixed $0$-section $S_0$}.

 \subparagraph{Changing section of elementary s.c. bundle vs cyclic shift of word, associated necklace.} \label{imath}
Let us have besides the $0$-section $S_0$ another $0$-section ${S_0}'$  of $\mf e$. The $0$-section ${S_0}'$ corresponds to a vertex of $F_t$ having number $\imath(S_0,{S_0}')$ in the vertex order determined by the orientation and section $S_0$. Then the word $\mc W(\mf e, {S_0}')$ is obtained from $\mc W(\mf e, S^0)$ by the action of cyclic shift $\tau_n^{\imath(S_0,{S_0}')} \mc W (\mf e, S_0) = \mc W (\mf e, {S^0}')$.

Thus the elementary s.c. bundle $\mf e$ obtains \textit{associated oriented  necklace  $\mc N(e)$} in the sence of \S\ref{neck} as an orbit of words $\mc W(\mf e, S_0)$ under cyclic shifts.

\subparagraph{Boundary of elementary s.c. bundle with section vs boundary of word. } \label{yy}
Take a face of the base simplex $\la k-1 \ra \xar{\la \delta_j\ra} \la k \ra$, and consider the corresponding boundary $\delta_j^* \mf e \xar{(\delta_j)_*} \mf e$ of the elementary s.c. bundle over that face. If $0$-section $S_0$ is fixed for $\mf e$ then $\delta_i^* e$ has induced $0$-section $\delta^*_i S_0$. So we have a morphism
$(\delta_j^*\mf e, \delta_j^*S_0 ) \xar{(\delta_j)_*} (\mf e, S_0)$. We are interested in relation between $\mc W=\mc W(\mf e, S_0)$ and $\mc W (\delta_j^*\mf e, \delta_j^*S_0 )$.
All $1$-sections $B$ of $\delta_i^* \mf e$ are  faces of $1$-sections $A_i$ of $\mf e$ such that $\mc W(i)\neq j$.

For every $1$-section $B$ of $\delta^*_j \mf e $ there is a unique $1$-section $\partial_{\delta_j} B$ of $\mf e$ such that $B$ is a face of $\partial_{\delta_j} B$ and this map is monotone respectively to induced order corresponding to the induced section $\delta_i^* S_0$. This statement is correct but requires some simplicial work. One should use mentioned in
\S\ref{cyl} $1$-dimensional version of Cohen's theorem on cylinder of simple map for PL manifolds.

  Combining the previous observations, we conclude: the
word $\mc W(\delta^*_j \mf e, \delta^*_j S_0) \equiv \delta^* \mc W(\mf e, S_0 )$ and the induced map on $1$-sections $\partial_{\delta_i} $ is domain boundary operator (\ref{wbound}) of words induced by codomain boundary $\delta_i$.

\subparagraph{Cyclic diagram of words associated to s.c. bundle with fixed sections on elemantary subbundles.} \label{cdw}
We denote by $\int^\NN \mf B$ the category of simplices of locally ordered simplicial finite complex $\mf B$. The objects are simplices and the only morphism between a $m$-simplex $U^m$ and a $k$-simplex $U^k$ is the face map $U^m \xar{\delta} V^k$ corresponding to face operator $[m]\xar{\delta}[k]$. This category has canonical functor $\int^\NN \mf B \xar{} \NN$ making it a category fibered in finite sets.

Let us have a s.c.  bundle $\mf E \xar{\mf p} \mf B$ over  $\mf B$. Let any elementary s.c. subbundle  $\mf p_U$ over a simplex $U \in \mf B$ has fixed  individual $0$-section $S_0^U$. We denote this system of $0$-sections $\bs S_0$; so we fix the pair $(\mf p, \bs S_0)$. Combining the constructions of \S\S\ref{word},\ref{imath},\ref{yy}, we see that
any boundary  $U \xar{\delta} V$ between simplices of $\mf B$ gives
a cyclic morphism of words
\begin{equation}\label{rot}
\mc W(\mf p, \bs S_0)(U \xar{\delta} V)=\left( \mc W(\mf p_U, S_0^U)
\xar{\tau^{ \imath(S_0^U, \delta^* S_0^V)}}
\mc W(\mf p_U, \delta^* S_0^V)  \xar{ \partial_\delta}\mc W(\mf p_V, S_0^V)\right).
\end{equation}
These data evidently commute with composition. Therefore we obtain a fibered over $\NN$ functor $\int^\NN \mf B \xar{\mc W(\mf p,\bs S^0)} C \mc W$. We can imagine that this functor is a coloring of the base simplices by words in alphabet consisting of vertices of the simplex, so that boundary morphisms on the base simplices correspond to cyclic morphisms of words in a way that diagram is commutative. Changing the $0$-sections causes equivalence of functors, i.e the system of cyclic permutations of words commuting with all the structure morphisms. Thus the isomorphism class of bundles goes to the equivalence class of fibered functors $\on{Funct}(\int^\NN \mf B, C \mc W)$. The inverse statement is also true with important comments (\S \ref{notes1}) but we don't fully develop it here.

\section{Rational parities of words and local formulas for powers of Chern classes} \label{rplf}
\subsection{Rational parity  of word and of odd necklace}
Consider a word $[n]\xar{w}[k]$. Call ``proper subword of $w$'' any subword  consisting of $k+1$ different letters (i.e. it is a section of map $w$). Proper subword  defines permutation of $k+1$ elements and this permutation has parity, even or odd. We define rational parity of $w$ to be the mathematical expectation of parities of all its proper subwords. Namely put
$$ P(w) = \frac{\#(\text{even proper subwords})-\#(\text{odd proper subwords})}{\#(\text{all proper subwords})} $$
Parity of a permutation of odd number elements is invariant under cyclic shifts of the permutation. Therefore if $k$ is even then $k+1$ is odd and in this situation $P(w)$ is an invariant of oriented necklace (\S\ref{neck}).  Words are ordered and has no nontrivial automorphisms. Necklaces can have remarkable groups of automorphisms \cite{Duzhin2013}. Parity of proper subword  survives  orientation preserving automorphism. Therefore if $k+1$ is odd, then $P(w)$ is a isomorphism invariant of the necklace, defined by $w$.

Parity of  permutation coincides with value of determinant of matrix of permutation.  Similar  interpretation has rational parity $P(w)$. Rational parity of $w$ is equal to the sum of maximal minors of normalized matrix of word $\ol{\mc L}$ defined in \S\ref{mat}:
\begin{equation}\label{rmatpar}
 P(w) = \sum_{0 \leq i_0<i_1<...<i_k \leq n} \det \ol{\mc L}_{(i_0,...,i_k)}(w)
\end{equation}
Where $\ol{\mc L}_{(i_0,...,i_k)}(w)$ denotes square submatrix of $\ol{\mc L}(w)$ with row numbers $i_0,...,i_k$.

\subsection{Local formulas for powers of Chern class in terms of rational parities of odd necklaces.}
Rational simplicial local formulas for Chern classes of circle bundles in the sense of \S\ref{GM} can be expressed trough rational parities of necklaces associated with elementary s.c. bundles (\S \ref{imath}):
\begin{theorem} \label{thm}
Rational function
\begin{equation}\label{npar}
\mbox{}^p\!C^h_1 (\mf e) = (-1)^h\frac{h!}{(2h)!} P(\mc N (e))
\end{equation}  of elementary s.c. bundle $\mf R \xar{\mf e} \la 2h \ra$	
over $2h$-simplex is a rational simplicial local formula for the $h$-th-power of the first Chern class of circle bundle.
\end{theorem} Here suffix $p$ stays for ``parity".

\section{Geometric simplicial, circle and cyclic bundles} \label{geometry}
\subsection{Metric of geometric proportions and associated  circle bundle.}
\subparagraph{Elementary degeneracy and geometric proportions} \label{gp}
Let us introduce the simplest possible metric on the geometric total complex $|\mf R|$ of elementary s.c. bundle $\mf R \xar{\mf e}\la k \ra$ satisfying the properties, described in p. \ref{ccb}. Here we essentially rely on facts about geometric proportions between  elements of similar triangles from book VI of Euclid's Elements, where he presumably explicates achievements of Pythagorean or Athens school. Actually, geometric proportions is a geometric background of Milnor's geometric realization functor.
The fact is following (see Fig. \ref{simplex} for similar triangles).
\begin{quote}
Consider a standard geometric  simplicial degenracy $A=\Delta^{k+1} \xar{|\sigma_j|} \Delta^k  =B$. It projects the edge $\alpha_j $ of $A$ with vertices $v'_j,v'_{j+1}$ to vertex $v_j$.
Let us fix a flat Eucledean metric $\rho$ on total simplex $A$, in which $\alpha_j$ has positive length $a_j \in \R_{>0}$. Take a point with barycentric coordinates $t=(t_0,...,t_k)\in B$ and consider the interval $L_j(t) = |\sigma_i|^{-1} (t) \subset A$, the  fiber of projection $|\sigma_j|$. Then the length $l^\rho_j(t)$ of interval $L_j(t)$ in metric $\rho$ is equal to $a_j\cdot t_j$.
\end{quote}
The domain simplex $A$ of the projection $|\sigma_j|$ obtains bundle coordinates if we declare one of two $0$-sections in it (see figure \ref{simplex}) as ``bottom" or ``tail" section and another as ``top" or ``head" section respectively. The  point $u \in A$ is given bundle coordinates $u=(x^{st}(u); t_0(u),...,t_k(0) )$, where $t(u)$ are baricentric coordinates of projection on base, $x^{st}(u)$ distance in fiber interval $L_j(t(u))$ from tail point to $u$ in standard flat metric $st$ on simplex,  $0 \leq x^{st}(u) \leq t_j(u)$. If we change flat metric by assigning to edge $\alpha_j$ length $a_j$ then the coordinate $x^{st}(u)$ will change lineary,  to $a_j \cdot x^{st}(u)$.

\begin{figure}[h]
	\centering
	\includegraphics[height=3.0in]{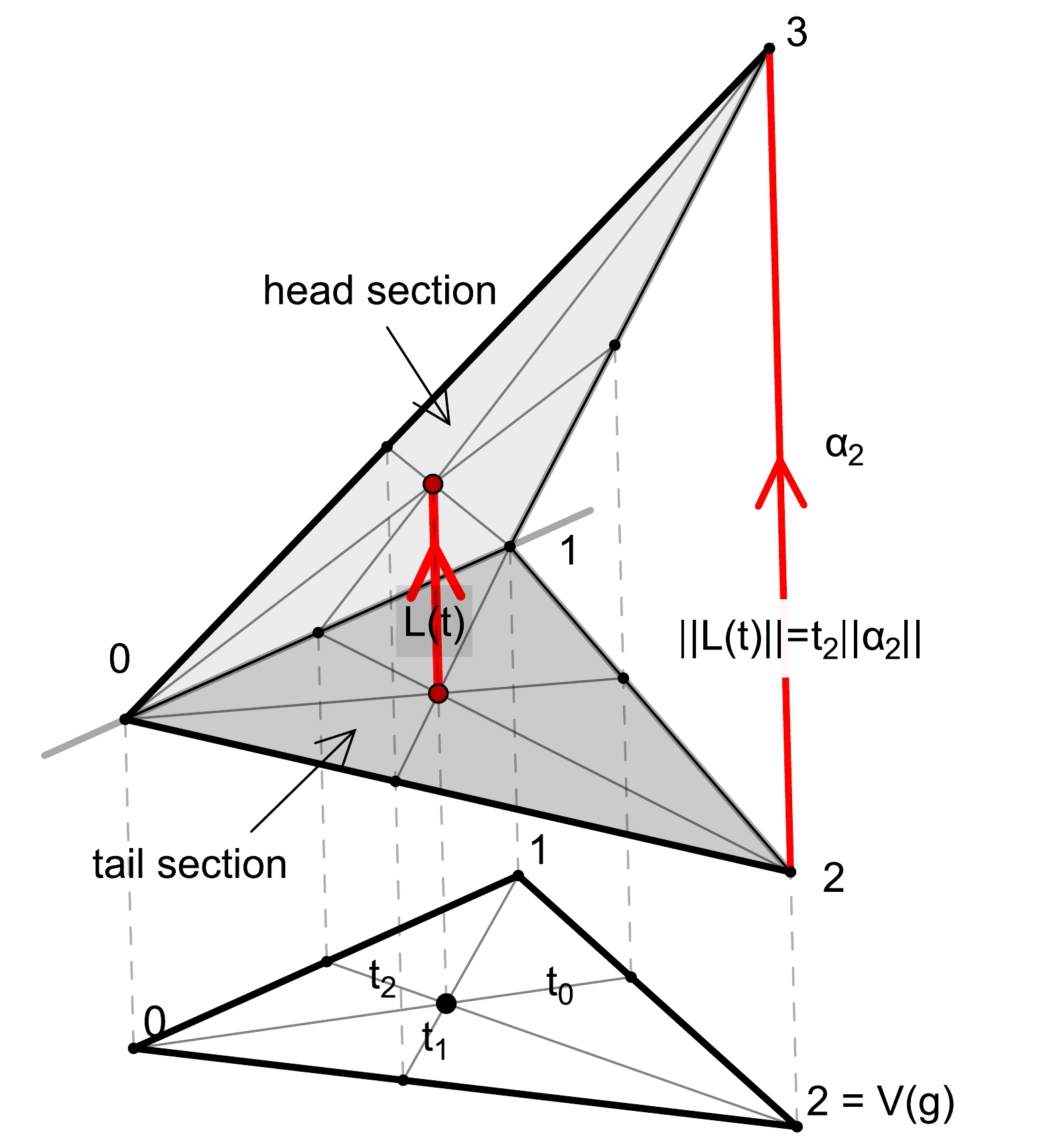}
	\caption{Bundle $\Delta^3 \xar{| \sigma_2| } \Delta^2$}\label{simplex}
\end{figure}
\subparagraph{Martrix of word and global bundle coordinates}
Geometric proportions provide geometric meaning to matrix $\mc L(\mc W)$ (\ref{mw}) \S\ref{mat} of word $\mc W=\mc W (\mf e, S^0)$  constructed in \S \ref{word} for elementary s.c. bundle $\mf R \xar{\mf e} \la k \ra$ with fixed $0$-section $S^0$. By geometric proportion identities \S\ref{gp},
\begin{equation}\label{key}
\sum_{a =0 }^k\mc L_i^a t_a = t_{\mc W(i)} =  l^{st}_i(t),
\end{equation}
 where $l^{st}_i(t)$ is the length of interval $L_i(t) = |e|^{-1} (t)\cap A_i$ in standard flat metric on $|\mf R|$.
 So, viewed as linear operator, matrix $\mc L$ transforms
 $$\mc L(t_0,...,t_k)=(l^{st}_0(t),...,l^{st}_n(t)) $$
 vector of barycentric coordinates of point $t$ in the base to vector of lengths of intersections with $1$-sections of fiber over $t$, ordered by bundle orientation.

The bundle space $|\mf R|$ also obtains global bundle coordinates relatively to chosen $0$-section $S_0$. The point $u$ lying  in a $1$-section $A_i$ is given the bundle coordinates
$$ u=(x^{st}_{\mc W(i)}(u) + l^{st}_0(t(u))+....+l^{st}_i(t(u); t(u)).$$
The $i$-th (with respect to orientation and fixed section $S_0$) $0$-section $S_i$ becomes the graph of function $S_i(t)=\sum_{b=0}^{i-1} l^{st}_b(t)$.

\subparagraph{Circle bundle coordinates} \label{cbc}
Let us normalize the standard flat metric on $|\mf e|$. Any elementary s.c. subbundle of $\mf e$ over a vertex $v_j$ of the base is a simplicial oriented circle $\mf R_j\xar{\mf e_{v_j}} \la 0\ra, |\mf R_j| \approx S^1$. Let the circle $\mf R_j$ have $m_j$ arcs. We assign to any arc in $|\mf R_j|$ the length equal to $\frac{1}{m_j}$. This induces new flat metric on all 1-sections $A_i$ since the collapsing edge of $A_i$ belongs to $|\mf R_{\mc W(i)}|$. So we obtain new flat metric ``$gp$'' on $|\mf R|$ (``$gp$'' stays for ``geometric proportions'').  In the new metric $gp$ vector of lengths of the intervals in the fiber over $t\in \Delta^k$ is expressible using normalized matrix of word $\ol{\mc L}(\mc W)$ ((\ref{rmw})\S \ref{mat}). Namely
\begin{equation}\label{key}
\sum_{a =0 }^k\ol{\mc L}_i^a t_a = \frac{1}{m_{\mc W(i)}}t_{\mc W(i)} =  l^{gp}_i(t),
\end{equation}

Then
\begin{equation}\label{reducedL}
\ol{\mc L}(t_0...t_k) = (l^{gp}_0(t),...,\ol l^{gp}_k(t))
\end{equation} and we have the identity
$$\sum_{i=0}^n l^{gp}_i(t) = \sum_{j=0}^k \frac{m_j}{m_j}t_j \equiv 1.$$
These observations show that $\ol{\mc L}$ is linear operator
\begin{equation}\label{Loperat}
\Delta^k \xar{\ol{\mc L}} \Delta^n.
\end{equation}
We have normalised the metric on $|\mf R|$ in such a way that all fiber circles would obtain unit lengths and face maps are isometries.

Now any point $u$ in $|\mf R|$, lying in the $i$-th $1$-section $A_i$ of $|\mf R|$, obtains trivial $\T$-circle bundle coordinates  $$ \mc T^{gp}_{S_0}(u) = (\exp(\frac{1}{m_{\mc W(i)}} x^{st}_{\mc{W}(i)}(u)  + \ S^{st}_{i}(t))| t) \in \T \times \Delta^k.$$ Let $\pmb T^k$ be a trivial circle bundle $\T \times \Delta^k \xar{} \Delta^k$. The map $\mc T^{gp}_{S_0}$ is a bundle  homeomorphism
$$|(\mf e, S_0)| \xar{\mc T^{gp}_{S_0}}(\pmb T^k, 0) $$, sending
$S_0$ to $0$-section and section $S_i$ to the graph of $\T$-valued function on $\Delta^k$ $$\exp(S^{gp}_i(t)),$$ where
\begin{equation}\label{trans}
S^{gp}_i(t)=\sum_{a=0}^{i}{ l^{gp}_a(t)}.
\end{equation}
Thus introducing the metric of geometric proportions  on bundle $|\mf e|$ with fixed  $0$-section $S_0$ produced $\T$-trivialization, controlled by data of reduced matrix of word $\ol{\mc L} (\mc W(\mf e, S_0))$.
\subparagraph{$0$-sections as $\T$-transiton functions; PD circle bundle $\pmb T^{gp}(\mc W(\mf p, \bs S^0)).$} \label{Tgp}
Let us consider a s.c. bundle $\mf p$ with a system $\bs S_0$ of $0$-sections over simplices of the base as in  \S \ref{cdw}. Let us fix the lengths of edges of circles $\mf E$ over vertices of $\mf B$ as in \S \ref{cbc}; then the  geometric realizations of all elementary subbundles obtain geometric proportions metric and  $\T$-trivialzation relative to the fixed sections. Suppose that the boundaries $U^m \xar{\delta} V^k$ of simplices in the base of trivializations $\pmb T^m$ and $\pmb T^k$ are related
by $\T$, the gauge transformation determined by change of $0$-sections (see Fig. \ref{mp})
\begin{figure}[h!]
	\begin{center}
		\includegraphics[width=3.0in]{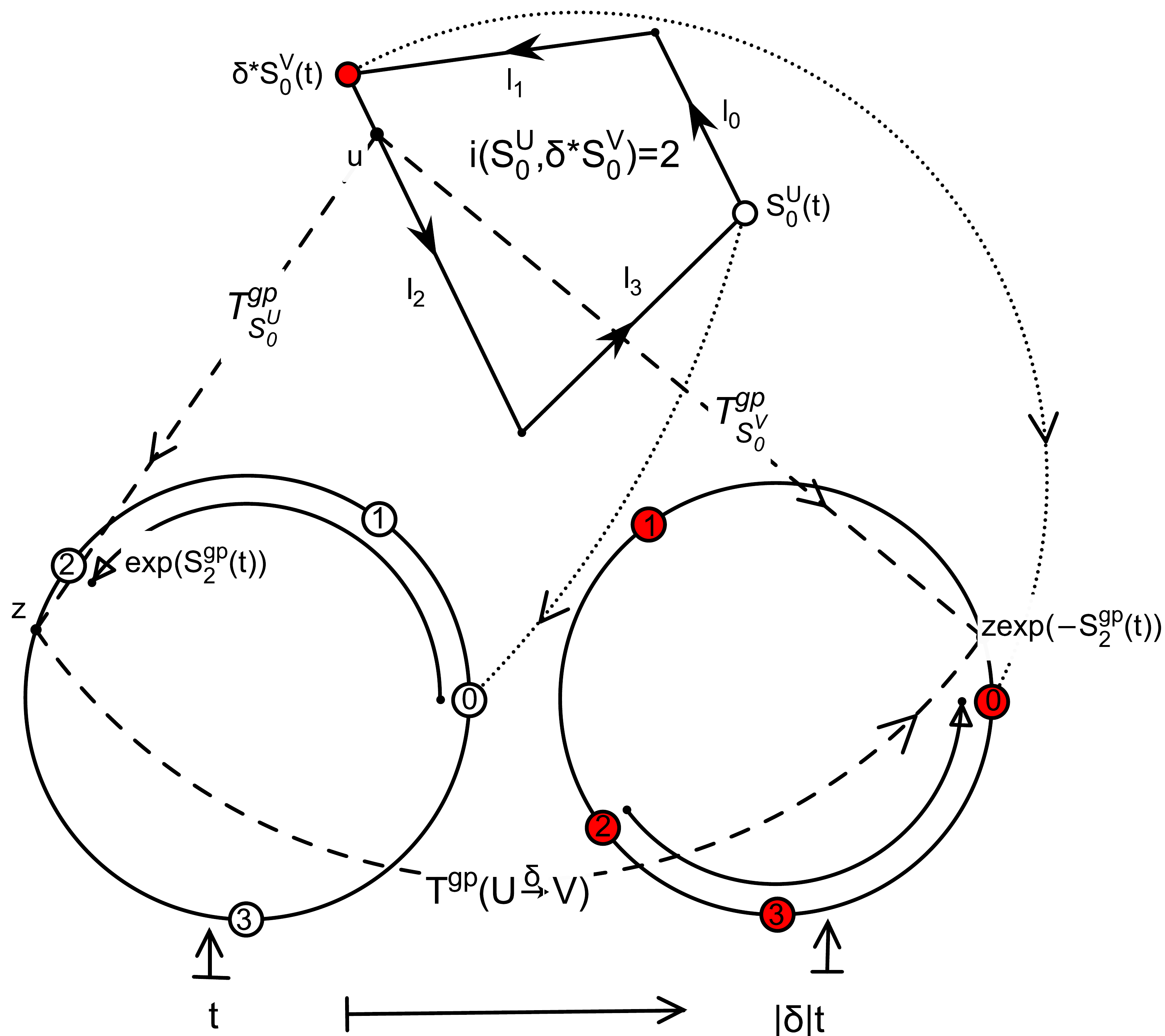} %\includegraphics[width=2.0in]{ebundl1.pdf}
		\caption{ \label{mp}}
	\end{center}
\end{figure}
followed by trivial face embedding:
\begin{equation}\label{tgp}
\xymatrixcolsep{3pc}\xymatrix{
{|(\mf p_U, S^U_0)|} \ar[r]^{\delta_{*}} \ar[d]_{\mc T^{gp}_{S_0^U}} &
{|(\mf p_V, S^V_0)|} \ar[d]^{\mc T^{gp}_{S_0^V}} \\
(\pmb T^m,0) \ar[r]^{\pmb{T}^{gp}(U\xar{\delta} V)} & (\pmb T^k,0)
}, \end{equation}
$$ \pmb T^{gp}({U\xar{\delta} V})(z|t) = (z\exp (-S^{gp}_{\imath(S_0^U, \delta^* S_0^V)}(t))| |\delta|(t))$$

Here $\imath$ was defined in \S\ref{imath}. Transition transformation $\pmb T^{gp}({U\xar{\delta} V})$ is determined by the cyclic morphism of words $\mc W(\mf p, \bs S_0)(U \xar{\delta} V)$ (see (\ref{rot})\S\ref{cdw}); moreover the diagram of trivial circle bundles and linear boundary gauge transformations is a function of $\mc W(\mf p, \bs U_0)$.

The diagram $\pmb T^{gp} (\mc W(\mf p, \bs S_0))$ assembles to a PD circle bundle over $|\mf B|$ and the map $\mc T^{gp}$
assembles to the triangulation $|\mf p|\xar{\mc T^{gp}} \pmb T^{gp}(\mc W(\mf p, \bs S_0))$. Changing  the system of sections $\bs S_0 \xar{} \bs S'_0$ causes global gauge transformation of $\pmb T^{gp}(\mc W(\mf p, \bs S_0) ) \xar{} \pmb T^{gp}(\mc W(\mf p, \bs S'_0) )$.

\subsection{Cyclic cosimplex and universal cyclic circle bundle}
We equip half of Connes' cyclic cosimplex \cite[Ch 7.1.3.]{Lo} with a $\NN C$-diagram of trivial $\T$ - bundles $\pmb T C$, having cyclic linear transition maps. The diagram $\pmb T C$ does not itself admit a colimit assembled to any sort of fiber bundle, but still we can use it as a universal object for fiber bundles $\pmb T^{gp}(\mf p, \bs S_0)$.

\subparagraph{Circle bundle over half of cyclic cosimplex} Let $\NN \xar{\triangle} \bs{Top}$ be canonical semi-cosimplex with barycentric coordinates  Below we will slightly change notations for boundary operators and coordinates on it and denote:
$$\triangle ([n]) = \Delta^n = \{l_0,...,l_n| \sum l_i = 1\} \subset \R^{n+1};$$
$$\triangle (\partial_i) = |\partial_i|(l_0,...,l_{n-1}) = (l_0,...l_{i-1},0,l_{i},...,l_{n-1}).$$
Cyclic group $\Z/(n+1)\Z$ acts by cyclic permutations of $[n]$ with standard generator $\tau_n$ represented by the right shift by one, $\tau_n (i) = i-1 \mod (n+1)$.

\textit{Cyclic semi-cosimplex} is a functor $\NN C \xar{\Delta C} \mb {Top}$; it is an extension of $\triangle$, that incorporates cyclic shifts of barycentric coordinates. For $[n]\xar{\tau^i_n}[n]$ the map $\triangle C (\tau^i_n)=(\Delta^n \xar{|\tau_n^j|} \Delta^n)$ is defined by standard coordinate representation $\Z/(n+1)\Z$ in $\R^{n+1}$:
$$|\tau^j_n| (l_0,...,l_n) = (l_{\tau_n^j(0)},...l_{\tau_n^j(n)}) = (l_{n+1-j}...,l_n,l_0....l_{n-j}).$$

Consider the following family of linear functions
$$ \Delta^n \xar{S_i^n} [0,1],i=0,...,n; S_i^n = \sum_{a=0}^{i-1}l_a.$$
Then the family of $\T$-valued functions on simplices
$$\exp(S_i^n(l)), i=0,...,n$$
can be viewed as a family of sections of trivial circle bundle $\pmb T^n = (\T\times \Delta)^n \xar{} \Delta^n$ over $\Delta^n$.
Geometrically sections $\exp(S_i^n(l)), i=0,...,n$ at point $l=(l_0,...,l_n)\in \Delta^n$ decompose the fiber over each point, which is circle of unit perimeter $\T$ with a fixed zero point, into intervals $L_i(l)$ of lengths $l_i$, i.e. the intervals  $L_i(l) =[\exp(S_i^n(l)) , \exp(S_{i+1}^n(l)) ] \subseteq \T$.
This object can be regarded as Kontsevich's oriented  metric $n$-polygon $mp(l)$ (Fig. \ref{mp}), see \cite{Kontsevich92} of unit perimeter with fixed zero vertex.
\begin{figure}[h!]
	\begin{center}
		\includegraphics[width=4.0in]{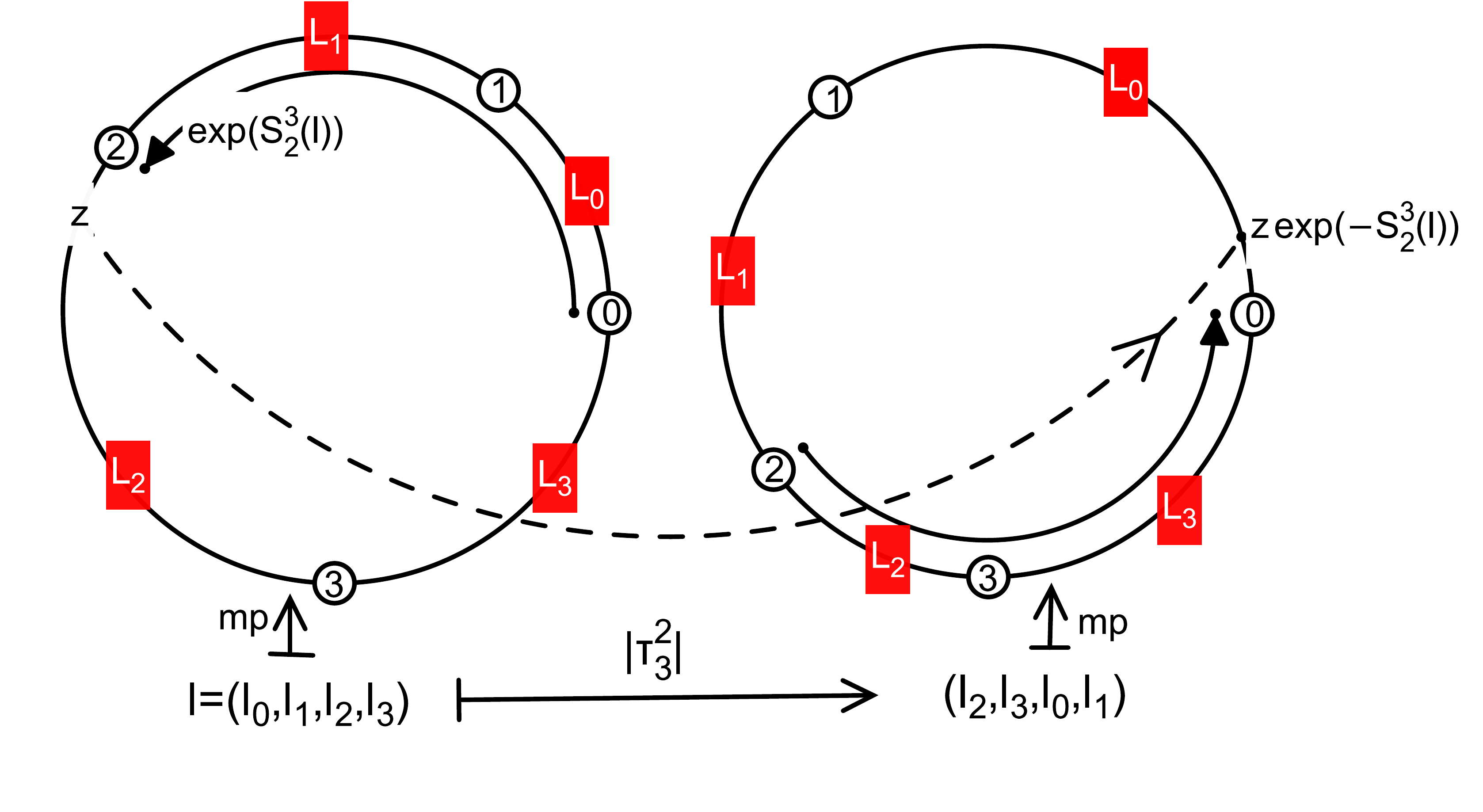} %\includegraphics[width=2.0in]{ebundl1.pdf}
		\caption{ $\pmb TC(\tau_3^2)$ \label{mp}}
	\end{center}
\end{figure}

Consider the diagram  $\pmb TC$  of trivial $\T$ - bundles and non-trivial gauge transformations  over cyclic semi-cosimplex  $\triangle C$. For $[n-1]\xar{\partial_i}[t]$ put
\begin{equation}\label{TCpartial}
\pmb T^{n-1} \xar{\pmb TC(\partial_i)} \pmb T^n : \pmb TC(\partial_i)(z|l)=(z||\partial_i|(l))
\end{equation}
For $[n]\xar{\tau_n^i} [n]$ put
\begin{equation}\label{TCtau}
\pmb T^n\xar{\pmb T C (\tau^i_n)} \pmb T^n : \pmb T C (\tau^i_n) (z| l) =
\left( z\cdot \exp( - S_i^n (l)) | |\tau^i_n| (l) \right)
\end{equation}
One can imagine the bundle maps $\pmb T C (\tau^i_n)$ being constructed as follows: first we rotate the fiber of $\pmb T^n$ over $l \in \Delta^n$ by the angle $\exp( - S_i^n (l))$, sending $\exp(S_i^n (l))$ to zero section, after which we make cyclic permutation $|\tau^i_n|$ of coordinates in the base, causing proper cyclic renumbering of the intervals and sections in the fiber (Fig \ref{mp}).

$\pmb T C $ is a correct $\NN C$ diagram of trivial $\T$-bundles over simplices  and gauge transformations over the cyclic semi-cosimplex $\triangle C$.  This fact is subject of checking composition law vs cyclic identities.

\subparagraph{Circle bundle $\pmb T^{gp}(\mc W(\mf p, \bs S^0))$ is a pullback of  universal bundle $\pmb T C$ by map composed from matrices of words  $\ol{\mc L}(...)$.} \label{cycpullback}

Circle bundle $\pmb T^{gp}(\mc W(\mf p, \bs S^0))$ was defined in \S\ref{Tgp}, its linear $\T$-transiton maps are given by (\ref{tgp}), using boundary change of $0$-sections in geometric proportions metric. What remains of geometric constructions now, it is just to observe that the bundle $\pmb T^{gp}(\mc W(\mf p, \bs S^0))$ can be canonically regarded as a pullback of the universal cyclic bundle  $\pmb T C$ so that the linear operators of normalized matrices of words $\ol{\mc L}(...)$ compose into the classifying map $\ol{\mc L}{\mc W(\mf p, \bs S_0)})$.

Let us make few simple observations on linear operators that come from the reduced matrices of words. Consider boundary morphism and cyclic shift of words
\begin{equation*}
\xymatrix{
{[n_0]} \ar[d]_{w_0} \ar[r]^{\partial_\delta} &{[n]} \ar[d]^w  \\
{[k_0]} \ar[r]^\delta &  {[k]}}
\xymatrix{
{[n]} \ar[rr]^{\tau_n^i} \ar[dr]_{w_0} & & {[n]} \ar[dl]^w \\
&{[k]} &
}
\end{equation*}
Then the following diagrams of linear maps in barycentric coordinates are commutative
 \begin{equation}\label{Ltrans}
 \xymatrix{
 \Delta^{n_0} \ar[r]^{|\partial_\delta|} & \Delta^n \\
  \Delta^{k_0} \ar[u]^{\ol{\mc L}(w_0)} \ar[r]^{|\delta|} & \Delta^k \ar[u]_{\ol{\mc L}(w)}
  }
  \xymatrix{
  \Delta^n \ar[rr]^{|\tau_n^i|}  & & \Delta^n \\
  & \Delta^k \ar[ul]^{\ol{\mc L}(w_0)}\ar[ur]_{\ol{\mc L}(w)} &
}
 \end{equation}
These diagrams induce pullback diagrams of boundary gauge transformations
 \begin{equation}\label{Lpull}
 \xymatrixcolsep{4pc}\xymatrix{
 \pmb T^{n_0} \ar[r]^{\pmb TC(\partial_\delta)} & \pmb T^n \\
 \pmb T^{k_0} \ar[u]^{\ol{\mc L}_*(w_0)} \ar[r]^{ \pmb T_\delta} & \pmb T^k \ar[u]_{\ol{\mc L}_*(w)}
 }
\xymatrix{
	\pmb T^{n} \ar[r]^{\pmb TC(\tau_n^i)} & \pmb T^n \\
	\pmb T^{k} \ar[u]^{\ol{\mc L}_*(w_0)} \ar[r]^{ \pmb T^{gp}(\tau^i_n)} & \pmb T^k \ar[u]_{\ol{\mc L}_*(w)}
}
\end{equation}
Assume that we are given a coloring of locally ordered simplicial complex $\mf B$ by cyclic diagrams of words (\S\ref{cdw}) $\int^\NN \mf B \xar{\mc W} C\mc W$.\\
We have two functors  $\int^\NN \mf B \xar{\triangle \mc W, \triangle \mf B } \bs{Top} $:
$$\triangle \mc W=\left( \int^\NN \mf B \xar{\mc W} C\mc W \xar{\on{Dom}} \NN C\xar{ \triangle C} \bs{Top}\right).  $$
$$\triangle \mf B = \left( \int^\NN \mf B \xar{} \NN \xar{ \triangle} \bs{Top}\right). $$
The diagrams (\ref{Ltrans}) are the data of natural transformation  $\triangle \mf B \xar{\ol{\mc L}} \triangle \mc W$ and the diagrams (\ref{Lpull}) are the data of pullback $\pmb T\mc W \xar{\ol{\mc L}_*} \pmb TC $ of bundle diagram $\pmb TC$ to PD circle  bundle $\pmb T \mc W$ on $|\mf B|$, defined by transition functions encoded in $\mc W$. Therefore, starting from $\mc W (\mf p, \bs S_0)$ we obtain circle bundle $\pmb T^{gp}(\mc W(\mf p, \bs S_0))$ as a pullback of $\pmb TC$ defined by normalized matrices of words, corresponding to elementary subbundles with fixed $0$-sections.

\section{Some linear algebra} \label{algebra}
\subsection{Universal cyclic invariant characteristic forms}
We wish to find a PD connection form on  $\pmb TC$. Observe, that the bundle $\pmb TC$ over cyclic semi-cosimplex $\triangle C$ strictly speaking is not a bundle, but a diagram of trivial $\T$-bundles over simplices and transition gauge transformations. So the natural definition for PD connection form on $\pmb TC$ is a family of connections  $\gamma_n \in \Omega^1\pmb T^n, n=0,1,...$ such that $$\pmb TC(\tau_n^i)^*\gamma_n = \gamma_n, i=0,...,n;\ \pmb TC(\partial)^*\gamma_n = \gamma _m \forall ([m]\xar{\partial}[n]).$$
One may call such a connection ``cyclic invariant connection''.   This property has for instance the connection  form ``$\alpha$'' on metric polygons \cite[p.8]{Kontsevich92} which we slightly recompile. If now $\gamma_n$ is a cyclic invariant connection, then any power of its curvature form $\omega^h = \wedge^h d\gamma_n$ should be be PD cyclic invariant form on $\triangle C$.
\begin{lemma} \label{konts}
Family of connection forms  $\alpha_n \in \Omega^1 \pmb T^n$ defined in local coordinates $(x; l_0,...,l_n)$ on $\pmb T^k = \T\times \Delta^k$, by the expression
\begin{equation}\label{con}
\alpha_n = -dx - \sum_{0\leq i< j \leq n}l_i dl_j
\end{equation}
is universal cyclic invariant. Its curvature form has expression
\begin{equation}\label{curv}
d \alpha_n = \omega_n = - \sum_{0\leq i< j \leq n}dl_i \wedge dl_j \in  \Omega^2\Delta^n.
\end{equation}
The power of curvature form has expression
\begin{equation}\label{hpow}
\omega^h_n = (-1)^h h!\sum_{0 \leq i_1< i_2<...<i_{2h} \leq n} dl_{i_1} \wedge   dl_{i_2} \dots\wedge dl_{i_{2h}} \in \Omega^{2h} \Delta^n.
\end{equation}

\end{lemma}
\begin{proof}
The form $\alpha_n$ is a Bott-style constructed universal form. The group of bundle automorphisms $\pmb T^n C(\Z/(n+1)\Z)$ of $\pmb T^n$ acts on connection forms lying in $\Omega^1 \pmb T^n$  and we wish to find a connection invariant under this action. We may try to construct invariant convex combination of the orbit of Maurer-Cartan horizontal connection $-dx$: for $i=1,...,n+1$ put $\pmb TC(\tau_n^i)^*(-dx) = -dx - dl_{k-i+1}-...- d l_n $ and take the smooth convex combination of connections in the orbit:
	\begin{equation}\label{}
	\begin{array}{l}
	+ \begin{array}{lll}
	l_n &\times &(-dx) \\
	l_{n-1}&\times &( -dx-dl_n) \\
	l_{n-2}&\times & (-dx-dl_{n-1} - d l_n)\\
	... &...& ...\\
	l_0 & \times & (-dx - dl_1 - ... -dl_n)\\
	\end{array}\\
%	\vspace{1mm}
%	---------------- \\
    \midrule
%    \hline
	-dx - \sum_{0\leq i< j \leq n}l_i dl_j
	\end{array}
	\end{equation}
The sum is a connection form $\alpha_n= -dx - \sum_{0\leq i< j \leq n}l_i dl_j$.
	
The fact that for any boundary $m \xar{\partial} n $: $\pmb TC(\partial)^* \alpha_n = \alpha_m$ follows from the definition of $\pmb TC(\partial)$ in $(z| l)$ coordinates (\ref{TCpartial}). We now need to check that $\alpha_n$ is cyclic-invariant: $\pmb T C(\tau_n^i)^* \alpha_n = \alpha_n$ (\ref{TCtau}). It is sufficient to ensure that it is true for the generator, i.e. that $\pmb T C(\tau_n)^* \alpha_n = \alpha_n$. We calculate:
	\begin{equation}
	\begin{array}{rl}
	\pmb TC(\tau_n)^* \alpha_n& = -dx -dl_n - l_n(dl_0 + ... + dl_{n-1}) - \sum_{0\leq i< j \leq n-1}l_i dl_j \\
	\alpha_n & = -dx -\sum_{0\leq i< j \leq n-1}l_i dl_j - (l_0+...l_{n-1})dl_n
	\end{array} \label{conn}
	\end{equation}
Taking in account that the coordinates are barycentric and substituting $l_0 = 1- l_1-...-l_n$ in (\ref{conn}) we obtain equal expressions in both cases:
	\begin{equation}
	\begin{array}{rl}
	\pmb TC(\tau_n)^* \alpha_n& = -dx -dl_n +l_n dl_{n} - \sum_{0\leq i< j \leq n-1}l_i dl_j \\
	\alpha_n & = -dx - \sum_{0\leq i< j \leq n-1}l_i dl_j - dl_n + l_{n}dl_n
	\end{array} \label{}
	\end{equation}
The transgression of $\alpha_n$ is a simplectic  $2$-form $\omega_n =d \alpha_n = -\sum_{0\leq i< j \leq n}dl_i \wedge dl_j$ on the base which is a $\NN C$-form since $\alpha$ is a $\pmb T C$-form; its pull-back to the bundle is the curvature of $\alpha_n$ (it's obvious that we can consider it instead of the curvature). The power of $\omega$ is obtained by a standard Garssmann algebra calculation as the power of Grassmann quadratic form, providing the factor $(-1)^h h!$.
\end{proof}

\subsection{Sum of minors - Pfaffian identity and ``matrix parity'' rational function.}
Let $X=X^{[n]\times[k]}$ be a $[n]\times[k]$ matrix of variables $x^i_j$, $i \in [n], j \in [ k] $. We suppose that $n \geq k$. Let $[{n+1 \atop k+1}]$ be the set of all $k+1$-element subsets of $[n]$. Let $D^{\mathbf a} \in \mathbb{Z}[X]$ be the maximal minor of $X$ with rows numbered by $\mathbf a \in [{n+1 \atop k+1}]$, considered as a polynomial. Consider the polynomial $ \pmb s = \sum_{\mathbf a \in [{n \atop k}]} D^\mathbf{a}(X) \in \mathbb{Z}[X] $, the sum of all maximal minors. Polynomial $\boldsymbol {s}$ can be expressed as a Pfaffian polynomial of an even skew-symmetric matrix in variables $X$. This is  Okada's sum of minors - Pfaffian identity \cite[Theorem 3]{Okada89},\cite{Ishikawa95}. Assume that $\mathbf u \subseteq [k]  $ and denote by $X_{\mathbf u}$ the submatrix of variables $X$ formed by columns with numbers in  $\mathbf u$ and by $\boldsymbol{s}_\bs{u} \in \mathbb{Z}[X]$ the sum of maximal minors polynomial for $X_{\mathbf u}$.

If $k+1$ is odd then
%======================
{\tiny }
\begin{equation} \label{odd} \boldsymbol{s} = \Pf
\begin{pmatrix}
0 & \boldsymbol{s}_{\{0\}}& \boldsymbol{s}_{\{1\}}&\boldsymbol{s}_{\{2\}}& \cdots & \boldsymbol{s}_{\{k\}}\\
-\boldsymbol{s}_{\{0\}}& 0 & \boldsymbol{s}_{\{0,1\}} & \boldsymbol{s}_{\{0,2\}}& \cdots & \boldsymbol{s}_{\{0,k\}}\\
-\boldsymbol{s}_{\{1\}}& -\boldsymbol{s}_{\{0,1\}}& 0 & \boldsymbol{s}_{\{1,2\}}& \cdots & \boldsymbol{s}_{\{1,k\}}\\
-\boldsymbol{s}_{\{2\}}& -\boldsymbol{s}_{\{0,2\}} & -\boldsymbol{s}_{\{1,2\}}& 0 & \cdots & \boldsymbol{s}_{\{2,k\}}\\
& & & \cdots & & \\
-\boldsymbol{s}_{\{k\}}& -\boldsymbol{s}_{\{0,k\}}&-\boldsymbol{s}_{\{1,k\}}&-\boldsymbol{s}_{\{2,k\}}&\cdots& 0
\end{pmatrix}
\end{equation}

If $k+1$ is even then

\begin{equation} \label{even} \boldsymbol{s} = \Pf
\begin{pmatrix}
0 & \boldsymbol{s}_{\{0,1\}} & \boldsymbol{s}_{\{0,2\}}& \cdots & \boldsymbol{s}_{\{0,k\}}\\
-\boldsymbol{s}_{0,1}& 0 & \boldsymbol{s}_{\{1,2\}}& \cdots & \boldsymbol{s}_{\{1,k\}}\\
-\boldsymbol{s}_{\{0,2\}} & -\boldsymbol{s}_{\{1,2\}}& 0 & \cdots & \boldsymbol{s}_{\{2,k\}}\\
& & \cdots & & \\
-\boldsymbol{s}_{\{0,{k}\}}&-\boldsymbol{s}_{\{1,{k}\}}&-\boldsymbol{s}_{\{2,{k\}}}&\cdots& 0
\end{pmatrix}
\end{equation}

We shall use the defining recursive identity for Pfaffian of skew-symmetric a $2m\times2m$ matrix $M$:
\begin{equation}
\Pf(M)=\sum_{j=2}^{2m} (-1)^j a_{1,j} \Pf(M_{\hat{1},\hat{j}}) \label{pf}
\end{equation}
where $M_{\hat{1},\hat{j}}$ denotes  the matrix $M$ with both the $1$-st and the $j$-th rows and columns removed.

Denote by $\delta^*_j X$ the $(n+1) \times k$ matrix obtained by deleting $j$-th column from $X$. Applying (\ref{pf}) to the r.h.s. of (\ref{odd}), and replacing the Pfaffians by the sums of minors from the l.h.s. of (\ref{odd}), (\ref{even}) we obtain the following identity if $k+1$ is odd:
\begin{equation} \label{par}
\boldsymbol{s}=\sum_{i=0}^{k}(-1)^{j} \boldsymbol{s}_{\{j\}} \boldsymbol{s}(\delta^*_j X)).
\end{equation}
Define a rational function of matrix, the ``matrix rational parity function'', by the formula
\begin{equation} \label{matpar}
P=\frac{\boldsymbol{s}}{\prod_{j=0}^{k} \boldsymbol{s_{\{j\}}}}\end{equation}
The important properties of the matrix rational parity are given in the following lemma:
\begin {lemma} \label{matrix_parity}
\mbox{}\\
a) If $k+1$ is odd then $P(X)$ is invariant under cyclic permutations of rows. \\
b)
\begin{equation}\sum_{j=0}^{k}(-1)^{j} P(\delta_j^* X) = \begin{cases} P(X) & \text{\rm if $k+1$ is odd,} \\ 0 & \text{\rm if $k+1$ is even.} \end{cases}\end{equation}
\end{lemma}
\begin{proof}
a) The determinant of odd dimensional matrix is invariant under cyclic permutations. Therefore if $k+1$ is odd then the sum of maximal minors of $X$ is invariant under the cyclic permutations of raws. \\
b) If $k+1$ is odd, then we can take expression (\ref{par}) and divide both sides by  $\prod_{i=0}^{k+1} \boldsymbol{s_j}$ resulting in required identity. If $k+1$ is even, then $\delta^*_j X $ has an odd number of columns; hence by the odd case we have a cocycle condition on the parity.
\end{proof}

\subsection{Pullback of the univesal cyclc characteristic forms by matrix map.}
\subparagraph{} \label{pull}
Let us have a $[n]\times [2h]$ matrix $A=\{a_i^j\},\ i\in 0,1,...,n, j\in 0,...,2h$ of nonnegative reals. We suppose that $n\geq 2h $ and $\sum_{i =0}^n a_i^j =1,\ j = 0,...,2h$. Consider $A$ as a linear map in barycentric coordinates $\Delta^{2h} \xar{A} \Delta^n$, where $t_0,...,t_{2h}$ are the baricentric coordinates on  $\Delta^{2h}$, and $l_0,...,l_n$ are coordinates on $\Delta^n$:
$$t= (t_0,...,t_{2h})\overset{A}{\mapsto}( l_0(t),...l_n(t)), l_i(t) = a_i^0 t_0+ a_i^1t_1+... + a_i^{2h}t_{2h}.$$
We wish to compute the pullback $A^*\omega^h$ of the $h$-th-power of the curvature (or transgresion) form (\ref{hpow})  in standard coordinates $t_0,t_1,...,t_{2h-1}$ on $\Delta^{2h}$. Denote by $\boldsymbol{s}(A)$ the sum of maximal minors  of matrix $A$.
\begin{lemma} \label{matrixpull}
	
	$$ A^*\omega_n^h= (-1)^h h!\boldsymbol{s}(A) dt_0 \wedge d t_1\wedge...\wedge dt_{2h-1}$$
\end{lemma}
\begin{proof}
We compute the summands in the the sum (\ref{hpow}):
	\begin{equation}\label{sum}
	\sum_{0 \leq i_1< i_2<...<i_{2h} \leq n} dl_{i_1}(t) \wedge   dl_{i_2}(t) \dots\wedge dl_{i_{2h}}(t),
	\end{equation}
corresponding to all $2h\times 2h+1$ submatrices of $A$ and then we apply identity of Lemma \ref{matrix_parity}. To describe a summand we first assume that $n=2h-1$ and compute $dl_{0}(t) \wedge   dl_{2} \dots\wedge dl_{2h-1}(t)$. Let $\delta^*_j A, j=0...2h$ be the square $2h\times 2h$ matrix which is obtained from $A$ by deleting the $j$-th column. We denote $\delta^*_jdt = dt_0\wedge dt_{j-1}\wedge dt_{j+1}\wedge...dt_{2h}$. Then by Grassmann algebra rules
	$$dl_{0}(t) \wedge   dl_{2}(t) \dots\wedge dl_{2h-1}(t) = \sum_{j=0}^{2h} \det( \delta^*_j A) \delta_j^*dt. $$
Substituting into the r.h.s. $t_{2h} = 1 - t_0 -...-t_{2h-1}$ we obtain
	\begin{equation} \label{qtbm}
	dl_{0}(t) \wedge   dl_{2}(t) \dots\wedge dl_{2h-1}(t) = \left( \sum_j (-1)^j\det( \delta^*_j A)\right) dt_0\wedge...\wedge dt_{2h-1}
	\end{equation}
Now we assume $n \geq 2h$ and apply (\ref{qtbm}) to every summand of (\ref{sum}). Finally using Lemma \ref{matrix_parity} in odd case and keeping in mind the condition that sums of elements in columns of $A$ is equal to $1$ (and hence the denominator in expression (\ref{matpar}) for matrix parity is equal to $1$) we obtain the desired expression for $A^*\omega^h_n$.

\end{proof}
\section{Proof of Theorem \ref{thm}} \label{proof}
We check that  $\mbox{}^p\! C_1^{2h}$ satisfies the definition of rational simplicial local formula from \S\ref{GM} \!.
\subparagraph{}
First, we check that $\mbox{}^p\! C_1^{2h}$  is a  a rational simplicial $2h$-cocycle on $\mf R^c(\s)$. Let $\mf N$ be the semi-simplicial set of isomorphism classes of necklaces. Then $\mf N_k$ (\S\ref{neck}) is the set of isomorphism classes of all finite necklaces having beads colored by $[k]$. Boundary map $\partial_i^* $ is induced from the corresponding boundary on words, i.e. by deletion of all beads with color $i$. Let $K^\bullet_\Delta(\mf N;\Q)$ be the rational simplicial cochain complex of $\mf N$. Then for a word $w \in C\mc W_{2h}$ (i.e. ``odd word'', word in alphabet of  $2h+1$ letters) rational parity $P(w)$ is an invariant of isomorphism class of oriented  ``odd necklace'' defined as cyclic orbit of $w$. Therefore rational  parity of odd necklaces is a function $\mf (\mf N_{2h} \xar{P^{2h}} \Q) \in K^{2h}_\Delta(\mf N;\Q)$. Rational cochain $P^{2h}$  is a simplicial cocycle, this follows from Lemma \ref{matrix_parity} (even case) applied to matrix representation of rational parity function (\ref{rmatpar}) and  matrix representation of  boundary of word \S\ref{mat} \!.  The association of necklace $\mc N(\mf e) $ to elementary s.c. bundle $\mf e$ (\S \ref{imath}) sends the isomorphim class of a bundle to the isomorphism class of necklace and boundary to boundary. Hence it defines a map of semi-simplicial sets $\mf R^c(\s) \xar{\mc N} \mf N$. So we get pullback $2h$ cocycle  $P^{2h}(\mc N(-))\in K^{2h}(\mf R^c;\Q)$. The $2h$-chain $\mbox{}^p\! C_1^{2h}$ (\ref{par}) is proportional to the cocycle  $P^{2h}(\mc N(-))$, therefore it is a rational simplicial $2h$-cocycle on $\mf R^c(\s)$.

\subparagraph{} We need to proof that for a s.c. bundle $\mf R\xar{ \mf p} \mf B $ pullback of $\mbox{}^p\! C_1^{2h}$ by map $\mf G_{\mf p} $ is a simplicial cochain on $\mf B$ representing $c_1^{2h}(\mf p)$. For the first Chern class the formula can be guessed and then checked on Madahar-Shakaria triangulation of Hopf bundle \cite{Madahar:2000}. For higher classes we can be sure only that $\mbox{}^p\! C_1^{2h}$ are some universal cocycles. We are not sure in homotopy class of $\mc N$ or $\mf R^c(\s)$, we have no good series  of examples of triangulated circle bundles to check. The last fact is related to well known problem of triangulation of the complex projective spaces $\mathbb CP^n$. It is very hard to triangulate $\mathbb CP^n$ (\cite{Sergeraert2010}), much harder to triangulate Hopf circle bundles over them.  Also it is diffucult to compare formulas with simplicial cup product of first class, this is related to the well known problems on formulas for cup product.
\subparagraph{} What we can do now, it is to use Chern-Weil homomorphism for Kontsevich's connection form $\alpha$ on metric polygons and then use De-Rham theorem. To this end, we have discussed cyclic bundle geometry in Section \ref{geometry} and linear algebra in Section \ref{algebra}.
\begin{enumerate}

\item Piecewise - differential Chern-Weil homorphism for piecewise - differential principal bundles exists as a by-product of Chern-Weil homomorphism for simplicial manifold principal bundles (\cite{Dupont1976}).

\item We choose system $\bs S_0$ of $0$-sections of elementary s.c. subbundles of $\mf p$, and obtain cyclic diagram of words $\mc W (\mf p, \bs S_0)$ on $\mf B$ (see \S\ref{cdw}).

\item We choose geometric proportions metric $gp$ on $|\mf E|$.  Normalized matrix of a word $\ol{\mc L}(\mc W(\mf R \xar{\mf e} \la k\ra, S_0))$ of elementary bundle $\mf e$ with fixed combinatorial section $S_0$ applied as linear operator to a point of the base simplex $\Delta^k $ produces vector of distances between $0$-sections of $|\mf R|$ in metric $gp$ ordered by orientation (see \S\ref{cbc}). Changing the section $S_0$ results in cyclic permutation of this vector. This is a point of communication between simplicial and cyclic geometry.

\item We associate with  $\mc W (\mf p, \bs S_0)$  PD circle bundle $\bs T^{gp}(\mc W(\mf E, \bs S_0))$ on $|\mf B|$ defined as a diagram of trivial bundles over simplices and transition gauge transformations defined by changing combinatorial sections (\S\ref{Tgp}). The circle bundle $\bs T^{gp}(\mc W(\mf p, \bs S_0))$ is canonically triangulated by $|\mf p|$.

\item In \S\ref{cycpullback} we obtained the bundle  $\pmb T^{gp}(\mc W(\mf p, \bs S_0))$ as a diagram pullback
$$
\bs T^{gp}(\mc W(\mf p, \bs S_0)) \xar{\ol{\mc L}_*} \pmb TC
$$
of the universal cyclic bundle diagram $\pmb TC$ over cyclic semi-cosimplex $\triangle C$. Diagram morphism $|\mf B| \approx \triangle \mf B \xar{\ol{\mc L}(\mc W(\mf p, \bs S_0))} \triangle C$ on simplex $\triangle(U^k)$ of the base is linear operator  $\Delta^k \xar{\ol{\mc L}(\mc W(\mf p_U, S^U_0))} \Delta^n$. Here $\mf p_U$ is the elementary s.c. subbundle of $\mf p$ over $U$, $S_0^U$ is its fixed section, and  $n+1$ is the number of combinatorial $0$-sections of $\mf p_U$, the same as the total number of  letters in the word $\mc W(\mf p_U, S^U_0)$.
\item  Cyclic invariance of Kontsevich's connection $\alpha$  on $\pmb TC$ (Lemma \ref{konts}  )  means that  its indivdulal pullbacks
 $$\alpha_U =\ol{\mc L}(\mc W(\mf p_U, S^U_0)^* \alpha \in \Omega^1(\pmb T^{gp}(\mc W(\mf p_U, S_0^U);\Q) \approx \Omega^1(\T \times \Delta^k;\Q),U \in \mf B$$
are invariant under changing the fixed section $S_0^U$ and therefore invariant under all transition gauge transformations. Hence the pullbacks $\alpha_U, U \in \mf B$ compose to a rational PD connection on PD circle bundle $\pmb T^{gp}(\mc W(\mf p, \bs S_0))$ invariant under all gauge tranformations caused by changing systems of sections $\bs S_0$. We can apply PD Chern-Weyl homomorphism and deduce that powers
$$
\omega^h_U =\ol{\mc L}(\mc W(\mf p_U, S^U_0)^* \omega^h   \in \Omega^{2h}(\triangle U;\Q)
$$
of the curvature $\omega_U = d\alpha_U \in \Omega^2(\triangle U;\Q)$ assemble into a rational PD form in  $\Omega^{2h}(|\mf B|;\Q)$ representing rational $h$-power of the first Chern class  $$c^{2h}_1(\pmb T^{gp}(\mc W(\mf p, \bs S_0));\Q) = c^{2h}_1(\mf p;\Q)\in H^{2h}(|\mf B|;\Q).$$
\item We can now apply the De Rham-Weyl-Dupont-Sullivan homotopy between $\Omega_{PD}(|\mf B|;\Q)$ and $K_\Delta(\mf B;\Q)$, obtaining the simplicial cocycles representing $c^h(\mf p;\Q)$, by integrating the forms $\omega^h_U$ over the base simplices. This gives  zero if dimension of the base simplex is not equal to $2h$. Thus we arrive at computing the pullbacks of the universal cyclic characteristic form
$$
\omega^h(\mc W(\mf e, S_0)) = \ol{\mc L}(\mc W(\mf e, S_0))^* \omega_n^h
$$
for elementary c.s. bundle $\mf e$ over $2h$-simplices having $n+1$ $0$-sections and integrating them over the base simplices. The form $\omega^h(\mc W(\mf e, S_0))$ is invariant under changing of the base section $S_0$, therefor the resulting number is an invariant of the necklace $\mc N (\mf e)$. The pullback of the cyclic form $\omega^h_n$  by matrix map on $\Delta^{2h}$ was computed in Lemma \ref{matrixpull} \S\ref{pull} using sum of minors - Pfaffian identity. The result is
$$
\omega^h(\mc W(\mf e, S_0)) = (-1)^h h! \bs s( \ol{\mc L}(\mc W(\mf e, S_0))dt_0\wedge...\wedge dt_{2h-1}.
$$
Here $s( \ol{\mc L}(\mc W(\mf e, S_0))$ is the sum of maximal minors of normalized matrix of odd word. This number is  equal to the rational parity of the necklace (\ref{rmatpar})
$$
s( \ol{\mc L}(\mc W(\mf e, S_0)) = P(\mc N(\mf e)).
$$
The factor $h!$ appears from the power of Grassmann quadratic form, $(-1)^h$ from the coordinate exchange rule for universal cyclic connection. What remains is to integrate the constant $2h$ form by $2h$ simplex which adds the volume $\frac{1}{2h!}$ of standard $2h$ simplex as factor and the promised local simplicial expression (\ref{npar}) for $c_1^h(\mf p; \Q)$ as $\mbox{}^p\! C_1^h(\mf e)$ is ready.

\end{enumerate}

\section{Notes}\label{notes}
\subparagraph{} \label{notes1}
Here we swept under the carpet an appropriate version of PL simplicial bundle theory. Although we need it only in elementary form and one-dimensional case, it still requires space for setup. Simplicial bundle theory is a parametric extension of simple homotopy theory of families of simple maps. It was presented  in \cite{waldhausen2013} and commented in lectures \cite{Lurie2014}.  Simple maps pop up in description of boundary of elementary s.c. bundle (\S\ref{yy}) and in 1-dimensional case  relate simplicial bundle combinatorics to cyclic category, this is what we are actually investigating.  In our case the adequate variant would be semi-simplicial which has not yet been fixed. Semi-simplicial bundle is a singular map of semi-simplicial sets (the same as  map of ``n.d.c. simplicial sets'' of \cite{RSI} or ``trisps map'' of \cite{Kozlov2008}). Semi-simplicial  circle bundles on given base  are in one-to-one correspondence with cyclic decorations of base by words.

\subparagraph{} Modulo the hidden semi-simplicial setup we
can formulate a couple of facts which we hope to write out somewhere in future.

Since Chern classes are integer classes, the corresponding simplicial cochains, represented by any rational local formulas,  should have integer simplicial periods, i.e. they are integrated to integer numbers over all integer $2h$ simplicial cycles in the base. When the base is some triangulation of an oriented closed surface this is a version of combinatorial Gauss-Bonnet theorem. This fact coupled with simple expressions $\mbox{}^p\!C_1(\mf e) $ provides some understanding  which bundles has or has not triangulation over particular simplicial base:
\begin{quotation}
	Let $|\mf B|$ be an oriented $2$-dimensional closed surface, triangulated by a classical simplicial complex $\mf B$ and the complex $\mf B$ has $F$ triangles. In this situation $F$ is always even number. Then Chern number of a classically triangulated circle bundle over $|\mf B|$ having $\mf B$ as the base complex, belongs to the integer interval $[-\frac{1}{2}F +1,...,\frac{1}{2}F-1]$
	
	Moreover Chern numbers of semi-simplicially  triangulated circle bundles over $|\mf B|$ having $\mf B$ as a base complex fill entire integer interval $[-\frac{1}{2}F,...,\frac{1}{2}F]$.  In this situatian $\mf B$ can be  assumed to be a finite semi-simplicial set, and $|\mf B|$ is ``$\Delta$-complex'' in the sense of \cite{hatcher2001algebraic}.
\end{quotation}
The only concrete example of a triangulated circle bundle observable in the literature is the triangulation of Hopf bundle over the boundary of tetrahedron $\partial \Delta^3$, constructed in \cite{Madahar:2000}. The parity local formulas allow one deduce that the cited result is the best possible. From the above statement one may conclude that over $\partial \Delta^3$ one can triangulate only trivial bundle and Hopf bundle using a map of classical simplicial complexes. If one can use semi-simplicial triangulations then over $\partial \Delta^3$ one can additionally triangulate circle bundle associated to tangent bundle of 2-sphere, and this is complete list of circle bundles allowing a triangulation over $\partial \Delta^3$. Classical triangulations are fundamental, but have their own additional degree of interesting arithmetical complexity relative to semi-simplicial triangulations, see \cite{JR1980}. Semi-simplicial category is related to classical simplicial by functorial double normal ( $\approx$ double barycentric) subdivision.

There is a somewhat strange more general statement which requires as a premise an integer combinatorial formula for the $1$-st Chern class:
\begin{quotation}
	If a circle bundle $p$ has triangultion with simplicial locally ordered base $\mf B$ then $c_1(p)$ can be represented by simplicial $2$-cocycle on $\mf B$ having values $0$ and $1$ on $2$-simplices.  The inverse is true for semi-simplicial triangulations of circle bundles and not true for triangulations by classical simplicial complexes.
\end{quotation}

\def\cprime{$'$}

%\bibliographystyle{halpha}
%\bibliography{combs}
\end{document}